\documentclass[a4paper,leqno]{article}

\usepackage[T1]{fontenc}
\usepackage[latin1]{inputenc}
\usepackage{mathtools,amssymb,amsthm,fullpage}
\mathtoolsset{mathic}
\usepackage{lmodern}
\usepackage{eucal}
\usepackage{hyperref}
\usepackage{microtype}
\usepackage{cleveref}
\usepackage{color}
\newlength{\alphabet}
\usepackage[textwidth=3\alphabet]{geometry}

\theoremstyle{plain}
\newtheorem{thm}{Theorem}[subsection]
\newtheorem{prop}[thm]{Proposition}
\newtheorem{coro}[thm]{Corollary}
\newtheorem{lem}[thm]{Lemma}
\newtheorem{ex}[thm]{Example}

\newcommand{\R}{\mathbb{R}}
\newcommand{\Z}{\mathbb{Z}}

\renewcommand{\H}{\mathbb{H}}
\newcommand{\C}{\mathbb{C}}

\newcommand{\Q}{\mathbb{Q}}
\newcommand*{\rom}[1]{\romannumeral}

\newcommand{\rot}{\mathrm{rot}}

\newcommand{\bma}{\begin{pmatrix}}
\newcommand{\ema}{\end{pmatrix}}
\newcommand{\LM}{\textrm{\rm LM}}

\providecommand{\msc}[1]{{\noindent\small\textbf{Mathematics Subject Classification (2020)} --- #1}}
\providecommand{\keywords}[1]{{\noindent\small\textbf{Keywords} --- #1}}

\DeclareMathOperator{\PSL}{PSL}
\DeclareMathOperator{\GL}{GL}

\DeclareMathOperator{\PSO}{PSO}

\DeclareMathOperator{\RE}{Re}
\DeclareMathOperator{\Tr}{Tr}
\DeclareMathOperator{\tr}{tr}
\DeclareMathOperator{\End}{End}
\DeclareMathOperator{\Vol}{Vol}

\DeclareMathOperator{\Id}{Id}
\DeclareMathOperator{\Ad}{Ad}

\DeclareMathOperator{\spec}{spec}

\DeclareMathOperator{\sign}{sign}

\DeclareMathOperator{\mult}{mult}

\numberwithin{equation}{section}
\numberwithin{thm}{section}

\renewcommand{\thethm}{
\ifnum\value{subsection}=0
\thesection
\else
\thesubsection
	\fi
		.\arabic{thm}
}

\title{Determinants of twisted Laplacians and the twisted Selberg zeta function\footnote{\keywords
hyperbolic orbisurface, twisted Selberg zeta function, non-unitary representation, Selberg trace formula, twisted Laplacian, regularized determinant}
\footnote{\msc Primary: 11M36; Secondary: 11F72, 37C30, 57Q10}}
\author{Jay Jorgenson\footnote{The first named author acknowledges grant support from PSC-CUNY Awards
67415-00 55 and 68462-00 56, which are jointly funded
by the Professional Staff Congress and The City University of New York.}, Lejla Smajlovic, Polyxeni Spilioti\footnote{The third named author was supported by the Hellenic Foundation for Research and Innovation (H.F.R.I.) under the "3rd Call for H.F.R.I. Research Projects to support Faculty Members $\&$ Researchers"(Project Number: 25622).}}

\begin{document}
\maketitle

\begin{abstract}\noindent
Let $X$ be an orbisurface, meaning a compact hyperbolic Riemann surface possibly with a finite number of elliptic points, and let $X_1$ denote
its unit tangent bundle. We consider the twisted Selberg zeta function  $Z(s;\rho)$ associated to a representation
$\rho \colon \pi_1(X_1) \to \GL(V_\rho)$.  We prove a relation between the \textit{twisted Selberg zeta function}  $Z(s;\rho)$
and the regularized  \textit{determinant of the twisted Laplacian} associated to $\rho$.  These results can be viewed as a
generalization of a result due to Sarnak \cite{sarnak1987determinants} who considered
the trivial character. Yet our proof is different, as it is based on evaluation of the Laplace-Mellin type integral transformations. Going further, we explicitly compute the multiplicative constant, which we call the \textit{torsion factor}, and express its dependence on parameters which determine the representation. We study the asymptotic behavior of the constant for a sequence of non-unitary representations introduced by Yamaguchi in \cite{Ya17} and prove that the asymptotic behavior of this constant as the dimension of the representation tends to infinity is the same as the behavior of the higher-dimensional Reidemeister torsion on $X_1$  (up to an absolute constant).

\end{abstract}

\bigskip

\section{Introduction}

\subsection{Regularized determinants as modular forms}

Let $\Lambda:=\{\lambda_{j}\}$ be the set of eigenvalues of a Laplacian $\Delta$ associated to some
geometric setting, such as the action on the space of smooth sections of a vector bundle on
a smooth compact Riemannian manifold.  The associated Hurwitz-type zeta function $\zeta(w,z)$
is (formally) defined as
$$
\zeta(w,z) = \sum\limits_{j}\frac{1}{(\lambda_{j}+w)^{z}}.
$$
Under general circumstances, it can be shown that $\zeta(w,z)$ defines a function of complex
variables $z$ and $w$ which lie in domain contained in a right-half plane, and, furthermore,
$\zeta(w,z)$ admits a meromorphic continuation which is regular at $z=0$.  With that,
one defines the regularized determinant $\det(\Delta +w)$ as
$$
\det(\Delta +w) := \exp\left(-\frac{d}{dz}\zeta(w,z)\bigg|_{z=0}\right).
$$
In a sense, one can view the regularized determinant $\det(\Delta +w)$ as a special value
of a zeta function, which itself is a field of mathematical significance.  However, this is just
the beginning since, for example, the regularized determinant $\det(\Delta +w)$ depends on whatever
moduli parameters are associated to $\Delta$, which introduces a rather promising means by which
one could study modular forms using spectral data.

For example, let $E_{\tau}$ be an elliptic curve with modular parameter $\tau$, and let $\mathcal{L}_{\chi}$ be a flat line bundle associated to the unitary character $\chi$.
Let $\Lambda$ be the set of eigenvalues of the flat Laplacian acting on smooth sections of $\mathcal{L}_{\chi}$.
Then it is shown in \cite{RS73} that $\det(\Delta) := \det(\Delta +w)\vert_{w=0}$ amounts to the Riemann theta
function with parameters $\tau,\chi$; see Theorem 4.1.  Since this result, numerous settings have been studied
where regularized determinants, and more generally analytic torsion, yield holomorphic modular forms.
We refer the reader to \cite{MR09} for specific examples involving Siegel modular forms and the Borcherds's $\Phi$ function,
both cases stemming from work due to Yoshikawa in \cite{Yo99} and \cite{Yo04}. These theorems lie within the vast field
of higher dimenesional Arakelov theory with its deep arithmetic significance.

\subsection{Regularized determinants and the Selberg zeta function}
Another interesting example from \cite{RS73} related spectral determinants to special values of the Selberg zeta function.
Let $X$ be a finite volume, smooth, genus $g$ hyperbolic Riemann surface, and let $\mathcal{L}_{\chi}$ and $\mathcal{L}_{\chi'}$ be
two flat line bundles associated to the unitary characters $\chi$ and $\chi'$ of $\pi_{1}(X)$.  Let $\det(\Delta,\chi)$,
resp. $\det(\Delta,\chi')$ be the spectral determinants of the eigenvalues of the hyperbolic Laplacian acting on smooth
sections of $\mathcal{L}_{\chi}$, resp. $\mathcal{L}_{\chi'}$.  Finally, let $Z(s,\chi)$ be the Selberg
zeta function twisted by the unitary character $\chi$.  Then, in our normalization, Theorem 4.6 of \cite{RS73} states that
\begin{equation}\label{eq:torsion_Selbergzeta}
\det(\Delta,\chi)/\det(\Delta,\chi') = Z(1,\chi)/Z(1,\chi').
\end{equation}
In \cite{Fa81} the author used \eqref{eq:torsion_Selbergzeta} and proved the following fascinating result.  The
spectral determinant $\det(\Delta,\chi)$ admits an analytic continuation from $\chi \in (S^{1})^{2g}$ to $(\mathbf{C}^{*})^{2g}$
whose divisor determines the period matrix of $X$.  Hence, by the Torelli theorem, the extension of $\det(\Delta,\chi)$
to non-unitary characters $\chi$ determines the complex structure of $X$.

In analogy with the elliptic curve setting, it was asked on page 169 of \cite{RS73} if one can express
$\det(\Delta,\chi)$ in terms of Riemann's theta function.  A partial answer to this was given in \cite{Fa81} who related
the divisor of the extension of $\det(\Delta,\chi)$ to non-unitary $\chi$ in terms the divisor of a function involving
the Riemann theta function and certain Pyrm differentials, and the result lies at the center of the analysis
of loc. cit.;  see pages 114-115.  The above mentioned question from \cite{RS73} was answered in it entirety in \cite{Jo91},
where related results from other authors are discussed.

\subsection{Our main result}
In the present article we study regularized determinant $\mathrm{det}(\Delta_{\tau_m,\rho}^{\sharp} +s(s-1))$ associated
to the Bochner-Laplacian $\Delta_{\tau_m,\rho}^{\sharp}$.
Specifically, let $X$ be a compact hyperbolic Riemann surface possibly with a finite number of elliptic points,
and let $X_1$ denote its unit tangent bundle.  Let $Z(s;\rho)$ be the Selberg zeta function associated to a representation
$\rho \colon \pi_1(X_1) \to \GL(V_\rho)$.  \emph{It is not assumed that $\rho$ is unitary.} The image $\rho(u)$ of a loop $u$ in $X_1$ corresponding
to a complete clockwise rotation of unit vectors around the base point in $X$ equals $e^{-i\pi m} \mathrm{Id}_{V_\rho}$ for some $m\in(-1,1]$.
The corresponding
Selberg trace formula is developed in \cite{BFS23}, and in particular the authors state and prove an identity for the associated
heat kernel.  We study the Laplace-Mellin transform of the trace of the heat kernel and prove, in Theorem \ref{thm:main_theorem}
\begin{equation}\label{eq:main_formula}
\mathrm{det}(\Delta_{\tau_m,\rho}^{\sharp} +s(s-1))= Z(s;\rho)Z_I(s,\rho)Z_{\rm ell}(s;\rho)e^{\widetilde C},
\end{equation}
where the functions $Z_{I}(s;\rho)$ and $Z_{\rm ell}(s;\rho)$ stem from the identity and elliptic contributions
to the trace formula.  The constant $\widetilde C$, which we call the \emph{torsion factor}, is given by
\begin{equation}\label{eq. tilde C defn}
\widetilde{C}=-\mathrm{dim}(V_\rho)\chi(X)(2\zeta'(-1)-\log\sqrt{2\pi})+ \sum_{j=1}^{r}\frac{\log\nu_j}{2\nu_j}\sum_{k=1}^{\nu_j-1}\frac{\mathrm{Tr}
  (\rho(\gamma_j)^k) e^{i\pi k m/\nu_j}}{\sin^2\left(\frac{\pi k}{\nu_j}\right)}
\end{equation}
where the sum is taken over representatives $\gamma_j$ of the $r$ inconjugate
elliptic classes with orders $\nu_j$, $\chi(X)$ is the Euler characteristic of $X$, and $\zeta$ denotes the Riemann zeta function.
In the case $\rho$ is the trivial
character, we reprove by different means some of the main results of \cite{sarnak1987determinants}.

The functions $Z_{I}(s;\rho)$ and $Z_{\rm ell}(s;\rho)$ are given explicitly in terms of the classical Gamma function
and Barnes double Gamma function as well as the topological data of $X$ and the representation $\rho$;
see Theorem \ref{thm:main_theorem}.

As an example of our main theorem, we compute the constant $\widetilde{C}$ in setting considered by Yamaguchi in \cite{Ya22}; see Example \ref{ex:Yamaguchi}. In this case, we needed to
compute certain trigonometric sums whose values were not evident, but their evaluations came from the general results
in \cite{JKS24} which themselves are examples of heat kernel analysis on discrete tori.
Moreover, we prove that for a wider class of all $2N$-dimensional representations $\rho_{2N}$ than studied in \cite{Ya17} one has the limiting behavior
$$
\lim_{N\to\infty}\frac{\widetilde{C}}{2N}= -\chi(X)(2\zeta'(-1)-\log\sqrt{2\pi}).
$$
This is reminiscent of the main result of \cite{Ya17} which states that
$$
\lim_{N\to\infty}\frac{\log|\mathrm{Tor}(X_1,\rho_{2N})|}{2N}= -\chi(X)\log 2,
$$
where $\mathrm{Tor}(X_1,\rho_{2N})$ signifies the higher-dimensional Reidemeister torsion (see also \cite{Mu12} and \cite{M-FP14} for a similar result for hyperbolic 3-manifolds). For this reason, we believe that the name torsion factor
for the constant $\widetilde C$ is justified and that this term deserves further analysis, which will be the subject of a consequent paper.

In light of the results from the literature cited above, it is reasonable to expect interesting results
from the study of $\mathrm{det}(\Delta_{\tau_m,\rho}^{\sharp} +s(s-1))$ as a function of $\rho$.  For instance,
because $\mathrm{det}(\Delta_{\tau_m,\rho}^{\sharp} +s(s-1))$ is invariant when replacing $s$ by $1-s$, one
then gets from \eqref{eq:main_formula} the functional equation for the Selberg zeta function $Z(s;\rho)$.  Also,
Dietmar's notation of complete filtration from \cite{De19} seems amenable to the development of Artin-Venkov-Zograf
type factorization formulas in the case when $\rho$ is not irreducible.  The details of these assertions will
be investigated elsewhere.

Additionally, we note that many authors have studied regularized determinants and Selberg zeta functions in the
setting of non-compact quotients of hyperbolic spaces.  In this direction, additional considerations are needed in order
to define the regularized determinant because, for example, the corresponding heat kernel is not trace class.  We anticipate
that one can combine existing results in the literature with the approach presented in this paper and yield
analogous results.

\subsection{Outline of the paper}
The present paper is organized as follows.  In section \ref{sec:prelim} we establish notation and recall necessary background
material.  In section \ref{sec:STF} we recall the trace formula for the Bochner-Laplace operator, as proved in \cite{BFS23}.
In section 4 we apply the Laplace-Mellin transform to the theta function associated to the above mentioned heat kernel, and
we obtain a precise formula and asymptotic behavior of $\log \mathrm{det}(\Delta_{\tau_m,\rho}^{\sharp} +s(s-1))$;
see Lemma \ref{lemma cont of det}.  The Laplace-Mellin transform of the geometric side of the trace formula involves
several computations, most of which are given in section 5.  Our main theorem is proved in section 6.  In section 7
we further study ${\widetilde C}$, which we call the torsion factor, and, we give further evaluations for the
representations considered by Yamaguchi in \ref{ex:Yamaguchi}.  Finally, in section 8, we give specializations of our
main theorem, namely to the setting when there are no elliptic fixed points on $X$ as well as case when $s=1$.


\section{Preliminaries}\label{sec:prelim}

In this section, we introduce our geometrical setting, namely that of a
compact hyperbolic orbisurface $X$, its unit tangent bundles $X_1$ and its fundamental group $\pi_1(X_1)$.
We recall also some representation theory of the universal cover of $\PSL(2,\R)$.
At the end of this section, we recall the Barnes double gamma function, which will be used in
the proof of our main result.  The object is to provide a brief discussion in order to make the
article reasonably self-contained.  Specific references for further details are provided.

\subsection{Compact hyperbolic orbisurfaces}
\label{subsec:Orbi}
Consider the upper half plane
\[
\H^2:=\{z=x+iy:y>0\}.
\]
The group $G=\PSL(2,\R)$ acts on $\H^2$ by fractional linear transformations. This action is transitive and the stabilizer of $i\in\H^2$ is the maximal compact subgroup $K=\PSO(2)$, hence
\begin{equation*}
\H^2\cong G/K.
\end{equation*}
Let $\mathfrak{g}=~\mathfrak{sl}(2,\R)$ be the Lie algebra of $G$, and let $\mathfrak{g}=\mathfrak{k}\oplus\mathfrak{p}$ be the Cartan decomposition of $\mathfrak{g}$ with respect to the maximal compact subgroup $K$.
The restriction of the Killing form on $\mathfrak{g}$ to $\mathfrak{p}$ induces a $G$-invariant metric on $\H^2$ which is a multiple of the Poincar\'{e} metric,
which is
\begin{equation*}
ds^2=\frac{dx^2+dy^2}{y^2}.
\end{equation*}
Let $\Gamma\subseteq G$ be a cocompact Fuchsian group, meaning a discrete subgroup of $G$ such that the quotient $X=\Gamma\backslash \H^2=\Gamma\backslash G/K$ is compact. Then the Poincar\'{e} metric on~$\H^2$ induces a metric on $X$ that turns it into a compact hyperbolic orbisurface. Note that in this case, every non-trivial element in $\Gamma$ is either hyperbolic or elliptic.

Let $\widetilde{G}$ denote the universal cover of $G$ and write~$\widetilde{H}$ for the preimage of a subgroup $H\subseteq G$ under the universal covering map. Consider the following one-parameter families in $G$ (modulo $\pm I_2$):
$$
k_\theta=\begin{pmatrix}\cos\theta&\sin\theta\\-\sin\theta&\cos\theta\end{pmatrix},
\qquad a_t=\begin{pmatrix}e^{\frac{t}{2}}&0\\0&e^{-\frac{t}{2}}\end{pmatrix}, \qquad n_x=\begin{pmatrix}1&x\\0&1\end{pmatrix},
$$
where $\theta,t,x\in\R$. Denote by $\widetilde{k}_\theta$, $\widetilde{a}_t$ and $\widetilde{n}_x$ the unique lifts to the universal cover $\widetilde{G}$, so then
$$ \widetilde{K}=\{\widetilde{k}_\theta:\theta\in\R\}, \qquad \widetilde{A}=\{\widetilde{a}_t:t\in\R\} \qquad \mbox{and} \qquad \widetilde{N}=\{\widetilde{n}_x:x\in\R\} $$
are connected one-parameter subgroups of $\widetilde{G}$. Then the Iwasawa decomposition
$$ \widetilde{G} = \widetilde{N}\widetilde{A}\widetilde{K} $$
holds. Note that the center $\widetilde{Z}$ of $\widetilde{G}$ is given by
$$ \widetilde{Z} = \{\widetilde{k}_\theta:\theta\in\Z\pi\}. $$

Let us define the following two vector fields on $X_1 = \Gamma \backslash G$.
\begin{enumerate}
\item The (generator of the) geodesic flow $\mathfrak X$, which is induced by the infinitesimal action of the group $(a_t)_{t \in \R}$.
\item The rotation in the fiber vector field $\mathfrak X_\rot$, generated by the infinitesimal action of the group $(k_\theta)_{\theta \in \R}$.
\end{enumerate}

\subsection{The relation between $\pi_1(X_1)$ and $\pi_1(X)$}

In this section, we study the relation between $\pi_1(X_1)$ and $\pi_1(X)$.
For further details, we refer the reader to \cite[Section 2]{fried1986fuchsian} and
\cite[Section 1.2]{BFS23}.

The unit tangent bundle of $\H^2$ naturally identifies with $G=\PSL(2,\R)$. Indeed the map
$$
G \to \H^2
\,\,\,\,\,\textrm{\rm given by} \,\,\,\,\,
g \mapsto g\cdot \textrm i
$$
is a fibration whose fiber is a circle $S^1$ of directions of unit tangent vectors.
It follows  that the quotient $X_{1}:=\Gamma \backslash G$ is a compact 3-manifold which we will call the unit tangent bundle of $X$.
It follows that $X_{1}$ is a Seifert fibered 3-dimensional manifold, and this fibration induces the exact sequence
$$
1 \to \Z = \pi_1(\PSO(2)) \to \pi_1(X_1) \to \Gamma= \pi_1(X) \to 1
$$
where $\pi_1(X)$ denotes the orbifold fundamental group of $X$.

A compact orbisurface $X$ can have at most finitely many conical, or elliptic, points.  Let $r$ be the number
of such points, and if $r>0$ we denote these points  $x_1, \ldots, x_r$ if $r>0$.
Upon removing a small disc neighborhoods $D_1, \ldots, D_r$ about each point, one gets a compact surface $\overline X$ whose
boundary $\partial \overline X=\{\partial D_1, \ldots, \partial D_r\}$ is a union of disjoints circles and with fundamental group
$$
\pi_1(\overline X)= \langle a_1, b_1, \ldots, a_g, b_g, c_1, \ldots, c_r \mid \prod_i[a_i,b_i]\prod_j c_j = 1 \rangle.
$$
By capping off the boundary components $\partial D_i$ by orbidiscs, one obtains a presentation for the fundamental group of the orbisurface,
namely that
$$
\pi_1(X)= \langle a_1, b_1, \ldots, a_g, b_g, c_1, \ldots, c_r \mid \prod_i[a_i,b_i]\prod_j c_j = 1, \, c_j^{\nu_j}=1 \rangle,
$$
for some set of non-zero natural integers $\{\nu_j\}$ for $j= 1, \ldots, r$.
These integers are the order of the elliptic conjugacy classes in $\Gamma$ corresponding to the set of elliptic points $\{x_j\}$.

We denote by $\langle u\rangle$ the fundamental group of the fiber $\PSO(2)$ which identifies with the center of $\widetilde{G}$.
Specifically, the generator $u=\widetilde{k}_\pi$ identifies with a loop in $X_1$ corresponding to a complete clockwise rotation of unit vectors around the base point in $X$. It is proved in \cite[Section 2]{fried1986fuchsian} that one has the presentation that
\begin{multline}\label{presentation}
\pi_1(X_1) = \langle a_1, b_1, \ldots, a_g, b_g, c_1, \ldots, c_r, u \mid \\ [u,a_i] =1, [u,b_i]=1, [u,c_j]=1, \,  \prod_i[a_i,b_i]\prod_j c_j = u^{2g-2+r}, \, c_j^{\nu_j}=u \rangle.
\end{multline}

Let $\rho:\pi_1(X_1)\to\GL(V_\rho)$ be a finite-dimensional complex representation of $\pi_1(X_1)$.
As the notation suggests, it is not assumed that $\rho$ is unitary.
Since $u\in\pi_1(X_1)$ is central, $\rho(u)$ commutes with $\rho(\gamma)$ for all $\gamma\in\pi_1(X_1)$.
If $\rho$ is irreducible, $\rho(u)$ must be a scalar multiple of the identity by Schur's Lemma, say $\rho(u)=
\lambda\,I_n$.
Let $N = \operatorname{lcm}(1, \nu_1, \ldots, \nu_r)$
be the least common multiple of $1$ and the orders of the elliptic conjugacy classes.
By definition, the orbifold Euler
characteristic $\chi(X)$ of $X$ is given by
$$ \chi(X) =  2-2g+\sum_{j=1}^r\left(\frac 1 {\nu_j}-1\right) \in \Q. $$
By \cite[Lemma 1.2.1]{BFS23}, we have the following lemma.

\begin{lem}\label{lem:CenterActionIrredPi1X1}
Let $\rho\colon \pi_1(X_1) \to \GL(V)$ be an irreducible representation of dimension $n$. Then $\rho(u) =\lambda\, I_n$ with $\lambda^{Nn\chi(X)}=1$
\end{lem}

\noindent
In other words, $\lambda$ is a root of unity.

For an elliptic element $\gamma_j\in \pi_1(X_1)$ of order $\nu_j=M(\gamma_j)$, it is known that $\rho(\gamma_j^{\nu_j})=
\rho(\gamma_j)^{\nu_j}=\rho(u)=e^{-i\pi m} {\rm Id}_{V_\rho}$ for some rational number $m$;
see \cite{BFS23}, p. 1417 and Section 3.1 below.  Due to periodicity of the exponential function, we can and will assume
throughout this paper we that $m\in(-1,1]$.
Therefore, $\rho(\gamma_j)$ has $n=\dim(V_\rho)$ eigenvalues of the form $\exp\left(-\frac{2\pi i}{\nu_j}(\frac{m}{2} +\alpha_{jp})\right)$
where we may assume that $\alpha_{jp}$ is an integer lying in the set $\{0,1,\ldots,\nu_j-1\}$.

For a non-negative integer $\ell$, let us define quantities
$$
\alpha_j(\ell):=\sum_{p=1}^n \alpha_{jp}(\ell)
\,\,\,\,\,
\text{\rm and}
\,\,\,\,\, \tilde{\alpha}_j(\ell):=\sum_{p=1}^n \tilde{\alpha}_{jp}(\ell),
$$
where $\alpha_{jp}(\ell)$, respectively $\tilde{\alpha}_{jp}(\ell)$, is the residue of $\alpha_{jp}+\ell$, respectively
$-\alpha_{jp}+\ell$, modulo $\nu_j$ with $j=1,\ldots, r$.

Moreover, for $m\in(-1,1]$, with the notation as above let us define
\begin{equation}\label{eq. constant C_m def}
  C_m(j,\ell;\rho):=  -\frac{1}\nu_{j} \sum_{k=1}^{\nu_j-1}\mathrm{Tr}(\rho(\gamma_j)^k)\frac{ie^{i\pi k m/\nu_j}}{2\sin\left(\frac{\pi k}{\nu_j}\right)}e^{-i\pi k (2\ell+1)/\nu_j}
\end{equation}
and
\begin{equation}\label{eq. constant C_m tilde def}
  \tilde C_m(j,\ell;\rho):=  \frac{1}\nu_{j} \sum_{k=1}^{\nu_j-1}\mathrm{Tr}(\rho(\gamma_j)^k)\frac{ie^{i\pi k m/\nu_j}}{2\sin\left(\frac{\pi k}{\nu_j}\right)}e^{i\pi k (2\ell+1)/\nu_j}
\end{equation}

Let us observe that definition of quantities $\alpha(\ell)$ and $\tilde\alpha(\ell)$ is the same as in \cite{FiBook87},
pp. 66--67 with $d=\dim(V_\rho)$, $2k=m$ and the multiplier system $\chi$ of weight $2k$ instead of representation $\rho$. Therefore, one can express those quantities in terms of the action of the representation $\rho$ onto elliptic representatives $\gamma_j$, $j=1,\ldots, r$ as
\begin{equation}\label{eq. alpha l rep}
\alpha_j(\ell)=\frac{1}{2}\dim(V_\rho)(\nu_j-1) +\nu_jC_m(j,\ell;\rho)
\end{equation}
and
\begin{equation}\label{eq. alpha tilde l rep}
\tilde\alpha_j(\ell)=\frac{1}{2}\dim(V_\rho)(\nu_j-1) + \nu_j\tilde C_m(j,\ell;\rho).
\end{equation}
Let us observe here that $C_m(j,\ell;\rho), \tilde C_m(j,\ell;\rho)$ are rational numbers. Trivially, one also has that
\begin{equation}\label{eq. sum of Cm s}
\sum_{\ell=0}^{\nu_j-1}(C_m(j,\ell;\rho) \pm \tilde C_m(j,\ell;\rho) )=0,
\end{equation}
for all $j=1,\ldots, r$.

\subsection{Twisted dynamical zeta functions}

Every closed oriented geodesic $\gamma$ on $X$ lifts canonically to a homotopy class in $\widetilde{\Gamma}=\pi_1(X_1)$,
which we also denote by $\gamma$.  The conjugacy class $[\gamma]$ in $\widetilde{\Gamma}/\widetilde{Z}\simeq\Gamma=\pi_1(X)$
consists of hyperbolic elements.  Conversely, every such conjugacy class of hyperbolic elements contains exactly one
representative of a closed geodesic. A closed geodesic $\gamma$ on $X$ is called \emph{prime} if it cannot be written as
a non-trivial multiple of a shorter closed geodesic coming from $\pi_{1}(X)$.

As above, let $\rho:\widetilde{\Gamma}=\pi_1(X_1)\to\GL(V_\rho)$ be a finite-dimensional, complex representation. For $s\in\C$,
we define the twisted Selberg zeta function
$$
Z(s;\rho) = \prod_{\gamma\text{ prime}}\prod_{k=0}^\infty\det\left(\Id-\rho(\gamma)e^{-(s+k)\ell(\gamma)}\right).
$$
It is shown in \cite[Theorem 3.1]{fedosova2020meromorphic} that $Z(s;\rho)$ converges for $s$ in some right half plane of $\C$ and defines a holomorphic function on this half plane (see also \cite[Section 1.2]{Wo} for the torsion-free case). We note that the logarithmic derivative
$L(s;\rho)$ of $Z(s;\rho)$ is given by
$$ L(s;\rho):= \frac{d}{ds}\log Z(s;\rho) = \sum_\gamma\frac{\ell(\gamma)\tr(\rho(\gamma))}{2n_\Gamma(\gamma)\sinh(\frac{\ell(\gamma)}{2})}e^{-(s-\frac{1}{2})\ell(\gamma)}, $$
where the summation is over \emph{all} closed geodesics.
By \cite[Theorem 4.1.2]{BFS23}, we have the following result.

\begin{thm}[Meromorphic continuation of the Selberg zeta function]
\label{mero}
	The twisted Selberg zeta function $Z(s;\rho)$ has a meromorphic extension to the whole complex plane $\C$ with poles and zeros given by the formal product
	\begin{align*}
		\prod_{j=0}^\infty&\left(s-\frac{1}{2}-i\mu_j\right)\left(s-\frac{1}{2}+i\mu_j\right)\\&
		\times\prod_{n=0}^\infty\left(s+\frac{m}{2}+n\right)^{N_+(m,n)}\left(s-\frac{m}{2}+n\right)^{N_-(m,n)},
	\end{align*}
	where $(\frac{1}{4}+\mu_j^2)_{j\in\Z_{\geq0}}\subseteq\C$ are the eigenvalues of $\Delta_{\tau_m,\rho}^\sharp$, counted with
integral algebraic multiplicity, and
	\begin{align*}
		N_\pm(m,n) &= \frac{\Vol(X)\dim(V_\rho)}{4\pi}(\pm m+2n+1)\\&
		\pm i\sum_{[\gamma]\textup{ ell.}}\frac{\tr\rho(\gamma)}{2M(\gamma)\sin(\theta(\gamma))}e^{\pm2i\theta(\gamma)(n\pm\frac{m}{2}+\frac{1}{2})}.
	\end{align*}
\end{thm}

\subsection{Representation theory of $\widetilde{\PSL(2,\R)}$}
\label{subsec:RepTheory}

We recall here the principal series and the (relative) discrete series of the universal covering group of $\PSL(2,\R)$,  following \cite[Section 1]{Hof94}.

Let $\widetilde{M}=\widetilde{Z}$ denote the center of $\widetilde{G}$.
Then $\widetilde{N}\widetilde{A}\widetilde{M}$ is a parabolic subgroup of $\widetilde{G}$. The unitary dual of $\widetilde{K}\simeq\R$ is comprised of the unitary characters $\tau_m$ for $m\in\R$ as defined by
\begin{equation}
	\tau_m(\widetilde{k}_\theta) = e^{im\theta} \qquad (\theta\in\R).\label{eq:DefTauM}
\end{equation}
The restriction $\sigma_\varepsilon$ of $\tau_m$ to $\widetilde{M}$ only depends on $\varepsilon=m+2\Z\in\R/2\Z$.
With this notation, the unitary dual of $\widetilde{M}$ is given by all $\sigma_\varepsilon$ with $\varepsilon\in\R/2\Z$.

For $\sigma=\sigma_\varepsilon$ with $\varepsilon\in\R/2\Z$ and $s\in\C$,
we form the principal series representation $\pi_{\sigma,s}$ of $\widetilde{\PSL(2,\R)}$ on
$$
I_{\sigma,s} = \{f\in C^\infty(\widetilde{G}):f(na_tmg)=\sigma(m)e^{st}f(g)\}
$$
by a right translation action. The $\widetilde{K}$-types in $I_{\sigma,s}$ are spanned by the functions $\phi_m$ with $m+2\Z=\varepsilon$, where
$$
\phi_m(na_tk) = e^{st}\tau_m(k) \qquad (n\in\widetilde{N},t\in\R,k\in\widetilde{K}).
$$
Note that the Casimir element $\Omega=\frac{1}{4}(H^2+2EF+2FE)$ with
$$ H=\begin{pmatrix}1&0\\0&-1\end{pmatrix}, \qquad E=\begin{pmatrix}0&1\\0&0\end{pmatrix}, \qquad F=\begin{pmatrix}0&0\\1&0\end{pmatrix} $$
acts in $\pi_{\sigma,s}$ by $s(s-1)\Id$.

\paragraph{Unitary principal series.} For $s=\frac{1}{2}+i\lambda$, $\lambda\in\R$, the representation $\pi_{\sigma,s}$ extends to a unitary representation on the Hilbert space
$$ \left\{f:\widetilde{G}\to\C:f(na_tmg)=\sigma(m)e^{(i\lambda+\frac{1}{2})t}f(g)
\,\,\,
\textrm{\rm where}
\,\,\,\int_{\widetilde{M}\backslash\widetilde{K}}|f(k)|^2\,dk<\infty\right\},
$$
thus yielding the \emph{unitary principal series}. We denote the distribution character of this representation by $\Theta_{\sigma,\lambda}$.

\paragraph{(Relative) discrete series.} For $m>0$, the $\widetilde{K}$-types $\phi_{m'}$ for $m'\in\pm\{m,m+2,m+4,\ldots\}$ span an $\sl(2,\C)$-invariant subspace of $I_{\sigma_\varepsilon,\frac{m}{2}}$ where $\varepsilon=\pm m+2\Z\in\R/2\Z$, which can be completed to a Hilbert space of a
unitary representation of $\widetilde{G}$ which is called the \emph{(relative) discrete series}. We write $\Theta_{\pm m}$ for the
distribution character of this representation.
For more details about the invariant inner product on the (relative) discrete series,
we refer the reader to \cite[Section 1]{Hof94}.

\subsection{Barnes double gamma function}

In this section, we recall the definition of the Barnes double gamma function and provide some formulas that we will use in our computations.

The Barnes double gamma function is an entire function of order two of a complex variable $s$ and is defined by
$$
G\left(s+1\right)=\left(2\pi\right)^{s/2}\exp\left[-\frac{1}{2}\left[\left(1+\gamma\right)s^{2}+s\right]\right]
\prod_{n=1}^{\infty}\left(1+\dfrac{s}{n}\right)^{n}\exp\left[-s+\frac{s^{2}}{2n}\right]
$$
where $\gamma$ is the Euler constant. The function $G(s+1)$ has a zero of multiplicity $n$ at each point $-n \in \{-1,-2,\dots \}.$

For $s \notin -\mathbb{N}$, we have that (see \cite[p. 114]{FiBook87})
\begin{equation} \label{log der barnes gamma}
\frac{G'(s+1)}{G(s+1)} = \frac{1}{2}\log(2\pi) + \frac{1}{2} - s + s\psi(s),
\end{equation}
where $\psi(s)=\frac{\Gamma'}{\Gamma}(s)$ denotes the digamma function.
For $\RE(s)>0$,  $\log G(s+1)$ admits an asymptotic expansion as $s \rightarrow \infty$, which we cite from
\cite{FL01} or \cite[Lemma 5.1]{AD14}. For every positive integer $n$, we have that
\begin{equation}\label{asmBarnes}
\log G(s+1) = \frac{s^2}{2}\left( \log{s} - \frac{3}{2}\right) - \frac{\log{s}}{12} - s \, \zeta^{\prime}(0) + \zeta^{\prime}(-1) \>- \\
\sum_{k=1}^{n} \frac{B_{2k+2}}{4\,k\,(k+1)\,s^{2k}} +  h_{n+1}(s),
\end{equation}
where $h_{n+1}(s)$ is a holomorphic function in the right half plane $\RE(s)>0$
which satisfies the asymptotic relation $h_{n+1}^{(j)}(s) = O(s^{-2n-2-j})$ as $\RE(s)\to \infty$ for every integer $j\geq 0$, and where the implied constant depends solely on $j$ and $n$.

\section{The trace formula}

\label{sec:STF}
In this section, we introduce the twisted Bochner--Laplacian on sections of certain vector bundles over $X$,
after which we state a trace formula for the corresponding heat operator.

\subsection{Orbifold vector bundles over $X$}
\label{subsec:vect}

Recall that $\widetilde{\Gamma}$ denotes the preimage of $\Gamma\subseteq G$ in $\widetilde{G}$ and $\widetilde{\Gamma}\simeq\pi_1(X_1)$, and consider a finite-dimensional, complex representation $\rho:\widetilde{\Gamma}\to\GL(V_\rho)$. Let
$$ E_\rho = V_\rho\times_{\widetilde{\Gamma}}\widetilde{G}\to\widetilde{\Gamma}\backslash\widetilde{G}=X_1 $$
denote the associated flat vector bundle over $X_1$, meaning
$$
E_\rho = (V_\rho\times\widetilde{G})/_{(\rho(\gamma)v,g)\sim(v,\gamma^{-1}g),\,v\in V_\rho,\gamma\in\widetilde{\Gamma},g\in\widetilde{G}}.
$$
We equip $E_\rho$ with a flat connection $\nabla^{E_\chi}$. In general, this vector bundle does not define a vector bundle over $X=X_1/\widetilde{K}$,
both because $\widetilde{\Gamma}$ and $\widetilde{K}$ contain the center $\widetilde{Z}$ of $\widetilde{G}$ and because
$\rho$ is not necessarily trivial on $\widetilde{Z}$. In order to obtain a vector bundle on $X$, we have to twist this construction by a
character of $\widetilde{K}$ which is compatible with $\rho$ on $\widetilde{Z}$. Therefore, we assume that $\rho(u)=e^{-im\pi} \cdot \Id$
for some $m\in\Q$. Note that this assumption is automatically satisfied if $\rho$ is irreducible by \cref{lem:CenterActionIrredPi1X1}.

Let $\tau=\tau_m:\widetilde{K}\to\operatorname{U}(V_\tau)=\operatorname{U}(1)$ denote the character of $\widetilde{K}$ defined in \cref{eq:DefTauM},
so then $\rho(u)=\tau(u)^{-1}\cdot\Id_{V_\rho}$. We consider the homogeneous vector bundle
$$
E_\tau=\widetilde{G}\times_{\widetilde{K}}V_\tau\to \H=\widetilde{G}/\widetilde{K},
$$
meaning $E_\tau$ is the quotient space defined as
$$
E_\tau = (\widetilde{G}\times V_\tau)/_{(gk,v)\sim(g,\tau(k)v),\,g\in\widetilde{G},k\in\widetilde{K},v\in V_\tau}.
$$
The invariant inner product on $V_\tau$ induces a Hermitian metric and a metric connection $\nabla^{E_\tau}$ on $E_\tau$. By the same reason as above, the bundle $E_\tau$ does not factor through $X=\widetilde{\Gamma}\backslash\widetilde{G}/\widetilde{K}$.

In order to obtain an orbifold vector bundle on $X$ associated with $\rho$, we consider the tensor product representation $V_\tau\otimes V_\rho$ of $\widetilde{K}\times\widetilde{\Gamma}$ and let
$$ E_{\tau,\rho} = \widetilde{G}\times_{\widetilde{K}\times\widetilde{\Gamma}}(V_\tau\otimes V_\rho) \to \widetilde{\Gamma}\backslash\widetilde{G}/\widetilde{K}, $$
where $\widetilde{K}\times\widetilde{\Gamma}$ acts on $\widetilde{G}$ by $g\cdot(k,\gamma)= \gamma^{-1}gk$.  In other words,
$$ E_{\tau,\rho} = (\widetilde{G}\times(V_\tau\otimes V_\rho))/_{(g\cdot(k,\gamma),w)\sim(g,(\tau(k)\otimes\rho(\gamma))w),\,g\in\widetilde{G},k\in\widetilde{K},\gamma\in\widetilde{\Gamma},w\in V_\tau\otimes V_\rho} $$
Note that this action is not faithful, but the elements in $\widetilde{K}\times\widetilde{\Gamma}$ are
acting trivially on $\widetilde{G}$ are of the form $(z,z)$, $z\in\widetilde{Z}$.  Hence, the action also is trivial
on $V_\tau\otimes V_\rho$ by construction. We further remark that this does not define a topological vector bundle in
the usual sense but rather an orbifold vector bundle, because the fiber over a singular point is not a vector space but
instead is the quotient of a vector space by a finite group action. However, the orbisurface $X$ has a finite cover which
is a manifold, and the orbifold vector bundle is induced by a topological vector bundle, in the usual sense, on the cover.
Therefore, we will employ a slight abuse of notation and consider connections, metrics and differential operators on the bundle
by which we mean the
corresponding objects on the finite cover acting on sections which are invariant under the finite group of deck transformations.
A more thorough discussion of orbifold vector bundles can be found in the work of Shen and Yu~\cite{SY22}.
In particular, smooth sections of the bundle $E_{\tau,\rho}$ can be identified with smooth functions $f:\widetilde{G}\to V_\tau\otimes V_\rho$ such that
$$
f(\gamma gk) = \left[\tau(k)^{-1}\otimes\rho(\gamma)\right]f(g) \qquad (g\in\widetilde{G},k\in\widetilde{K},\gamma\in\widetilde{\Gamma}).
$$

Note that if $\rho$ is trivial and $X$ has no singularities, then the bundle $E_{\tau_m,1}\to X$ for $m\in2\Z$ is the bundle of Fourier modes of degree $m$
whose sections are identified with smooth functions $f : X_1 \to \C$ such that
$$ f(x,v) = f (x,R_\theta v)e^{i\pi m\theta} \qquad \mbox{for all }x\in X,v\in T_xX,|v|=1, $$
where $R_\theta$ is the rotation of angle $\theta$ in $T_xX$.

We now define a connection $\nabla^{E_{\tau,\rho}}$ on $E_{\tau,\rho}$ in terms of its associated covariant derivative. For this we identify vector fields $\mathfrak X\in C^\infty(X,TX)$ on $X$ with smooth functions $\mathfrak X:\widetilde{G}\to\mathfrak{p}$ such that $\mathfrak X(\gamma gk)=\Ad(k)|_{\mathfrak{p}}^{-1}\mathfrak X(g)$. Then,
$$
\nabla^{E_{\tau,\rho}}_{\mathfrak X}f(g) = \left.\frac{d}{dt}\right|_{t=0}f(g\exp(t\mathfrak X(g)))
\,\,\,
\textrm{\rm for all} f\in C^\infty(X,E_{\tau,\rho}
\,\,\,
\textrm{\rm and}
\,\,\,
\mathfrak X\in C^\infty(X,TX)
$$
defines a covariant derivative, hence a connection, on $E_{\tau,\rho}$. Note that, in general, there is no
Hermitian metric on $E_{\tau,\rho}$ compatible with the connection since the representation $\rho$ is not necessarily unitary.

\subsection{The twisted Bochner--Laplacian}
\label{subsec:Laplacians}

For a complex vector bundle $E\to X$ with covariant derivative $\nabla^E$, the second covariant derivative $(\nabla^E)^2$ is defined by
$$
(\nabla^E)^2_{\mathfrak X,\mathfrak X'} = \nabla_{\mathfrak X}^E\nabla_{\mathfrak X'}^E-\nabla^E_{\nabla_{\mathfrak X}^{\operatorname{LC}}\mathfrak X'}
\,\,\,
\textrm{\rm for all}
\,\,\,
\mathfrak X,\mathfrak X'\in C^\infty(X,TX),
$$
where $\nabla^{\operatorname{LC}}$ is the Levi--Civita connection on $TX$.
The negative of the trace of the second covariant derivative is the corresponding Bochner--Laplacian, which we write as
$$
\Delta_E = -\tr\left((\nabla^E)^2\right).
$$
For $\rho$ and $\tau$ as above, the twisted Bochner--Laplacian $\Delta_{\tau,\rho}^\sharp$ is defined as the Bochner--Laplacian of
the bundle $E_{\tau,\rho}$, meaning
$$
\Delta_{\tau,\rho}^\sharp = \Delta_{E_{\tau,\rho}} = -\tr\left((\nabla^{E_{\tau,\rho}})^2\right).
$$
Note that the orbisurface $X$ has a finite cover which is a manifold. Therefore, the spectral theory of $\Delta_{\tau,\rho}^\sharp$ on
$X$ is the same as the spectral theory of its lift to the finite cover when acting on sections that are invariant under the action of the finite
group of deck transformations. Hence, by \cite[Theorem 4.3]{Mk}, we have the following properties (see also for the manifold case
the work of M\"{u}ller~\cite{M1}, and for the orbifold case the work of Fedosova~\cite{fedell} and Shen~\cite[Section 7]{Shen2020}).

\begin{enumerate}
\item If we choose a Hermitian metric on $E_{\tau,\rho}$, then $\Delta_{\tau,\rho}^\sharp$ acts in $L^2(X,E_{\tau,\rho})$ with domain $C^\infty(X,E_{\tau,\rho})$. However, it is not a formally a self-adjoint operator in general. Note that while the inner product on $L^2(X,E_{\tau,\rho})$ depends on the chosen inner product on $V_\rho$, the space $L^2(X,E_{\tau,\rho})$ as a topological vector space does not, due
    to the compactness of $X$.
\item The operator $\Delta_{\tau,\rho}^\sharp$ is an elliptic second order differential operator with purely discrete spectrum $\spec(\Delta_{\tau,\rho}^\sharp)\subseteq\C$ consisting of generalized eigenvalues. The spectrum is contained in a
    translate of a positive cone in $\C$. (Here, a positive cone is a cone whose closure is contained in $\{z\in\C:\RE z>0\}\cup\{0\}$.)
    The generalized eigenspaces
$$ \{f\in L^2(X,E_{\tau,\rho}):(\Delta_{\tau,\rho}^\sharp-\mu\Id)^Nf=0\mbox{ for some }N\} $$
are finite-dimensional, contained in $C^\infty(X,E_{\tau,\rho})$, and their direct sum, as $\mu$ runs through
$\spec(\Delta_{\tau,\rho}^\sharp)$, is dense in $L^2(X,E_{\tau,\rho})$.
\end{enumerate}

\subsection{The trace formula}
\label{subsec:pretrace}

The heat operator $e^{-t\Delta_{\tau_m,\rho}^\sharp}$ corresponding to the twisted Laplacian $\Delta_{\tau_m,\rho}^\sharp$ is of trace class,
which follows from the Weyl Law; see \cite[Lemma 2.2]{M1}).  By \cite[equation (5.6)]{M1} the integral kernel of the heat operator is the smooth function
and can be realized as
$$
H_t^{\tau_m,\rho}(g_1,g_2) = \sum_{[\gamma]\subseteq\widetilde{\Gamma}/\widetilde{Z}} \rho(\gamma)\otimes H_t^{\tau_m}(g_2^{-1}\gamma g_1)
\,\,\,
\textrm{\rm for all}
\,\,\,
g_1,g_2\in\widetilde{G},
$$
where the summation is over the conjugacy classes of $\widetilde{\Gamma}/\widetilde{Z}\simeq\Gamma$.  The heat kernel of the Bochner--Laplace operator on $E_{\tau_m}\to\H$ is written as $(g_1,g_2)\mapsto H_t^{\tau_m}(g_2^{-1}g_1)$ with
$$ H_t^{\tau_m}\in(C^\infty(\widetilde{G})\otimes\End(V_{\tau_m}))^{\widetilde{K}\times\widetilde{K}}. $$
Here, $\widetilde{K}\times\widetilde{K}$ acts on $C^\infty(\widetilde{G})$ by left and right translation, on $\End(V_{\tau_m})$ by left and
right composition with $\tau_m$, and on their tensor product by the tensor product of these actions. Note that the center $\widetilde Z$ of
$\widetilde G$ lies in $\widetilde K$, hence each summand is independent on the choice of $\gamma$ relatively to $\widetilde Z$. The trace of $e^{-t\Delta_{\tau_m,\rho}^\sharp}$
can be computed in two different ways, by summing over the generalized eigenvalues of $\Delta_{\tau_m,\rho}^\sharp$ (the spectral side)
and by integrating the heat kernel along the diagonal in $\widetilde{\Gamma}\backslash\widetilde{G}$ (the geometric side). For the spectral side
we denote for an eigenvalue $\mu\in\spec(\Delta_{\tau_m,\rho}^\sharp)$ its algebraic multiplicity by
$$
\mult(\mu;\Delta_{\tau_m,\rho}^\sharp) = \dim\{f\in L^2(X,E_{\tau_m,\rho}):(\Delta_{\tau_m,\rho}^\sharp-\mu\Id)^Nf=0\mbox{ for some }N\}.
$$
Note that $\mult(\mu;\Delta_{\tau_m,\rho}^\sharp)$ is independent of the chosen metric on $E_\rho$. For the geometric side, we use the standard arguments grouping conjugacy classes in $\widetilde{\Gamma}/\widetilde{Z}$; see for example \cite[Section 2]{Wa}. By equating the two expressions, one
obtains the pre-trace formula, which was made rigorous by M\"{u}ller~\cite[Proposition 5.1]{M1} for manifolds and generalized to orbifolds by
Shen~\cite[Theorem 7.1]{Shen2020}.  Specifically,
\begin{align}\label{eq:PreTrace}
	\sum_{\mu\in\spec(\Delta_{\tau_m,\rho}^\sharp)}&\mult(\mu;\Delta_{\tau_m,\rho}^\sharp)e^{-t\mu}\\&\notag
	= \sum_{[\gamma]\subseteq\widetilde{\Gamma}/\widetilde{Z}}\Vol(\widetilde{\Gamma}_\gamma\backslash\widetilde{G}_\gamma)\tr{\rho(\gamma)}\int_{\widetilde{G}_\gamma\backslash\widetilde{G}}\tr H_t^{\tau_m}(g^{-1}\gamma g)\,d\dot{g},
\end{align}
where $\widetilde{G}_\gamma$ resp. $\widetilde{\Gamma}_\gamma$ denotes the centralizer of $\gamma\in\Gamma$ in $\widetilde{G}$ resp. $\widetilde{\Gamma}$, and the summation on the right hand side is over all conjugacy classes $[\gamma]$ of elements in $\widetilde{\Gamma}/\widetilde{Z}\simeq\Gamma$.

Since $\widetilde{\Gamma}/\widetilde{Z}\simeq\Gamma$ consists of the identity element $e$, hyperbolic and elliptic elements, we can
compute the three types of contributions to the right-hand side in \eqref{eq:PreTrace} separately. These computations have been carried
out in \cite[Sections 3.4, 3.5 and 3.6]{BFS23}.
The combination of these computations is given in \cite[Theorem 3.7.1]{BFS23} which is the following trace formula.

\begin{thm}[Selberg trace formula with non-unitary twist]\label{thm:TraceFormula}
	Let $\rho:\widetilde{\Gamma}\to\GL(V_\rho)$ be a finite-dimensional complex representation with $\rho(u)=e^{-im\pi}$, $m\in\R$.
Set
\begin{equation}\label{Aop}
 A_{\tau_m,\rho}^\sharp=\Delta_{\tau_m,\rho}^\sharp-\frac{1}{4}.
\end{equation}
Then for every $t>0$ we have that
	\begin{align}\label{tracefinal}
		& \Tr(e^{-tA_{\tau_m,\rho}^\sharp}) = \sum_{\mu\in\spec(\Delta_{\tau_m,\rho}^\sharp)}\mult(\mu;\Delta_{\tau_m,\rho}^\sharp)e^{-t(\mu-1/4)}
\\\notag
		& \quad=\frac{\Vol(X)\dim(V_\rho)}{4\pi}\Bigg[\int_\R e^{-t\lambda^2}\frac{\lambda\sinh(2\pi\lambda)}{\cosh(2\pi\lambda)+\cos(\pi m)}\,d\lambda\\\notag
		& \hspace{4cm}+ \sum_{\substack{1\leq\ell<|m|\\\ell\textup{ odd}}}(|m|-\ell)e^{(\frac{|m|-\ell}{2})^2t}\Bigg].\\\notag
		& \qquad+\frac{1}{2\sqrt{4\pi t}}\sum_{[\gamma]\textup{ hyp.}}\frac{\ell(\gamma)\tr\rho(\gamma)}{n_\Gamma(\gamma)\sinh(\frac{\ell(\gamma)}{2})}e^{-\frac{\ell(\gamma)^2}{4t}}\\\notag
		& \qquad+\sum_{[\gamma]\textup{ ell.}} \frac{\tr\rho(\gamma)}{4M(\gamma)\sin(\theta(\gamma))}\Bigg[\int_\R e^{-t\lambda^2}\frac{\cosh(2(\pi-\theta(\gamma))\lambda)+e^{i\pi m}\cosh(2\theta(\gamma)\lambda)}{\cosh(2\pi\lambda)+\cos(\pi m)}\,d\lambda\\\notag
		& \hspace{4cm}+2i\sign(m)\sum_{\substack{1\leq\ell<|m|\\\ell\textup{ odd}}}e^{i\sign(m)(|m|-\ell)\theta(\gamma)}e^{(\frac{|m|-\ell}{2})^2t}\Bigg].
	\end{align}
The summation above is over the conjugacy classes $[\gamma]$ of hyperbolic, resp. elliptic, elements in
$\widetilde{\Gamma}/\widetilde{Z}\simeq\Gamma$, and the representative $\gamma\in\widetilde{\Gamma}$ is
chosen such that it is conjugate in $\widetilde{G}$ to $\widetilde{a}_{\ell(\gamma)}$ for $\ell(\gamma)>0$,
resp. to $\widetilde{k}_{\theta(\gamma)}$ for $\theta(\gamma)\in(0,\pi)$.
\end{thm}

\section{The determinant of twisted Laplacian}

In this section we follow Part I of \cite{JL} and study the Hurwitz zeta function associated to the theta series
$$
\Theta^{\sharp}_{\tau_m,\rho}(t):= \Tr(e^{-t\Delta^{\sharp}_{\tau_m,\rho}})= \sum_{\mu\in\spec(\Delta_{\tau_m,\rho}^\sharp)}\mult(\mu;\Delta_{\tau_m,\rho}^\sharp)e^{-t\mu}.
$$

As stated above, the operator $\Delta^{\sharp}_{\tau_m,\rho}$ is an elliptic, second order operator with purely discrete spectrum,
hence so is $A^{\sharp}_{\tau_m,\rho}$. Therefore, by \cite{Gi95}, in the case of dimension $n=2$,
there exist constants $\alpha_k$ for integers $k\geq 0$ such that the short time asymptotic of the trace is given by
\begin{equation}\label{eq. short time asymp}
\Tr(e^{-tA^{\sharp}_{\tau_m,\rho}})\sim t^{-1}\sum_{k\geq 0}\alpha_kt^k
\,\,\,\,
\textrm{\rm as $t\to 0^+$.}
\end{equation}
The Hurwitz zeta function associated to the theta series $\Theta^{\sharp}_{\tau_m,\rho}(t)$ is the function
\begin{equation}\label{eq:Hurwitz}
\zeta^{\sharp}_{\tau_m,\rho}(z;s)= \sum_{\mu\in\spec(\Delta_{\tau_m,\rho}^\sharp)}\frac{\mult(\mu;\Delta_{\tau_m,\rho}^\sharp)}{(\mu + s(s-1))^z}
\end{equation}
for $s,z\in \C$ with $\RE((s-1/2)^2)$ and $\RE(z)\gg 0$.  It is elementary to show that
\begin{align}\label{Lmtr}
\zeta^{\sharp}_{\tau_m,\rho}(z;s)&=
\textrm{\rm LM}\big[\Tr(e^{-tA^{\sharp}_{\tau_m,\rho}})\big]((s-1/2)^2;z)\\&=\frac{1}{\Gamma(z)}\int_{0}^{\infty} e^{-t(s-1/2)^2}\Tr(e^{-tA^{\sharp}_{\tau_m,\rho}}) t^{z-1}dt\notag
\end{align}
where $\textrm{\rm LM}(f)(s,z)$ is the Laplace-Mellin transform of a function $f$.  In other words,
the Hurwitz zeta function \eqref{eq:Hurwitz}
is the Laplace-Mellin transform of $\Tr(e^{-tA^{\sharp}_{\tau_m,\rho}})$
in the "Laplace" variable $(s-1/2)^2$ and "Mellin" variable $z$, and further normalized with the factor $1/\Gamma(z)$.
\footnote{We note that
$\textrm{\rm LM}\big[\Tr(e^{-tA^{\sharp}_{\tau_m,\rho}})\big]((s-1/2)^2;z)=
\textrm{\rm LM}\big[\Tr(e^{-t\Delta^{\sharp}_{\tau_m,\rho}})\big](s(s-1);z).
$}
The asymptotic expansion \eqref{eq. short time asymp} implies that the integral \eqref{Lmtr} is well defined for $\RE(z)>1$.

If, for a fixed complex $s$, the function $\zeta^{\sharp}_{\tau_m,\rho}(z;s)$ possesses holomorphic continuation to $z=0$,
then for $\RE(s(s-1))\gg 0$, one can define the determinant $\mathrm{det}(\Delta^{\sharp}_{\tau_m,\rho} + s(s-1))$ of $\Delta^{\sharp}_{\tau_m,\rho}$ by
$$
-\log\mathrm{det}(\Delta^{\sharp}_{\tau_m,\rho} + s(s-1)):= \left.\frac{d}{dz} \zeta^{\sharp}_{\tau_m,\rho}(z;s)\right|_{z=0} =\left.\frac{d}{dz}\left( \textrm{\rm LM}\big[\Tr(e^{-t\Delta^{\sharp}_{\tau_m,\rho}})\big](s(s-1);z)\right)\right|_{z=0}
$$
which formally is equivalent to the statement that
$$
\mathrm{det}(\Delta^{\sharp}_{\tau_m,\rho} + s(s-1))=\prod_{\mu\in\spec(\Delta_{\tau_m,\rho}^\sharp)}(\mu+s(s-1))^{\mult(\mu;\Delta_{\tau_m,\rho}^\sharp)}.
$$

If $\mu=0$ is not an eigenvalue of $\Delta^{\sharp}_{\tau_m,\rho}$, and assuming that \eqref{Lmtr} possesses
an analytic continuation to $s=1$, then the determinant of the operator $\Delta^{\sharp}_{\tau_m,\rho}$ is defined as
$$
-\log\mathrm{det}(\Delta^{\sharp}_{\tau_m,\rho}):= \left.\frac{d}{dz} \zeta^{\sharp}_{\tau_m,\rho}(z;1)\right|_{z=0}
$$
If $\mu=0$ is an eigenvalue of $\Delta^{\sharp}_{\tau_m,\rho}$ with multiplicity $M_\rho=\mult(0;\Delta_{\tau_m,\rho}^\sharp)$,
then, formally one would have that
\begin{align*}
\mathrm{det}(\Delta^{\sharp}_{\tau_m,\rho} + s(s-1))&=[s(s-1)]^{M_\rho}\prod_{\mu\in\spec(\Delta_{\tau_m,\rho}^\sharp), \mu\neq 0}
(\mu+s(s-1))^{\mult(\mu;\Delta_{\tau_m,\rho}^\sharp)}\\&:= [s(s-1)]^{M_\rho} \mathrm{det}(\Delta^{\sharp\ast}_{\tau_m,\rho} + s(s-1))
\end{align*}
where $\ast$ means that the product is taken over non-zero eigenvalues. In this case, again formally one would have that
$$
 \left.\frac{d^{M_\rho}}{ds^{M_\rho}} \mathrm{det}(\Delta^{\sharp}_{\tau_m,\rho} + s(s-1))\right|_{s=1}=M_\rho! \mathrm{det}(\Delta^{\sharp\ast}_{\tau_m,\rho}).
$$
Hence, we define
\begin{equation}\label{eq. def ast det}
\mathrm{det}(\Delta^{\sharp\ast}_{\tau_m,\rho}):= \frac{1}{M_\rho!}\left.\frac{d^{M_\rho}}{ds^{M_\rho}} \left(\mathrm{det}(\Delta^{\sharp}_{\tau_m,\rho} + s(s-1))\right)\right|_{s=1},
\end{equation}
assuming that the right-hand side of the above equation is well-defined.

Let us now justify the above formal definitions. We start by proving the following lemma where we show that
for $\RE(s-1/2)^2\gg 0$, the function $\big[\Tr(e^{-t\Delta^{\sharp}_{\tau_m,\rho}})\big](s(s-1);z)$ possesses holomorphic continuation to
$z=0$. Moreover, we also deduce the asymptotic expansion of its derivative at $z=0$ as $s\to\infty$.

\begin{lem}\label{lemma cont of det}
For  $\RE(s-1)^2\gg 0$, function $\textrm{\rm LM}\big[\Tr(e^{-t\Delta^{\sharp}_{\tau_m,\rho}})\big](s(s-1);z)$
possesses an analytic continuation to $z=0$. Moreover, with the notation as in \eqref{eq. short time asymp} we have that
\begin{multline*}
\left. -\log\mathrm{det}(\Delta^{\sharp}_{\tau_m,\rho} + s(s-1))\right. = \left.\frac{d}{dz}\left( \LM\big[\Tr(e^{-t\Delta^{\sharp}_{\tau_m,\rho}})\big](s(s-1);z)\right)\right|_{z=0} \\=2\alpha_0(s-1/2)^2\log(s-1/2)-\alpha_0(s-1/2)^2 -2\alpha_1\log(s-1/2) + o(1),
\end{multline*}
as $s\to\infty$.
\end{lem}

\begin{proof}
As above, define the theta function
$$
\Theta^{\sharp}_{\tau_m,\rho}(t):= \Tr(e^{-t\Delta^{\sharp}_{\tau_m,\rho}})
$$
  Using \eqref{eq. short time asymp} for $\RE(s-1/2)^2\gg 0$ and $\RE(z)>1$ we can write that
\begin{multline}\label{eq. starting LM}
  \big[\Tr(e^{-t\Delta^{\sharp}_{\tau_m,\rho}})\big](s(s-1);z)= \\ \frac{1}{\Gamma(z)} \int_{0}^{\infty} e^{-t(s-1/2)^2}\left(\Theta^{\sharp}_{\tau_m,\rho}(t) - \frac{\alpha_0}{t} - \alpha_1\right) t^{z-1}dt\\
  +\frac{1}{\Gamma(z)}  \int_{0}^{\infty} e^{-t(s-1/2)^2}\left(\frac{\alpha_0}{t} + \alpha_1\right) t^{z-1}dt.
  \end{multline}
Therefore, we have that
\begin{equation}\label{eq. sum for cont}
 \LM\big[\Tr(e^{-t\Delta^{\sharp}_{\tau_m,\rho}})\big](s(s-1);z)=I_1(s,z)+\frac{\alpha_0 (s-1/2)^{2-2z}}{(z-1)}+\alpha_1(s-1/2)^{-2z},
\end{equation}
where $I_1(s,z)$ denotes the first integral on the right-hand side of equation \eqref{eq. starting LM}. In view of the small time asymptotic expansion \eqref{eq. short time asymp}, and having in mind that absolute values of eigenvalues of $\Delta^{\sharp}_{\tau_m,\rho}$ tend to infinity, it is immediate that $I_1(s,z)$ is holomorphic at $z=0$ and so are the other two terms on the right-hand side of \eqref{eq. sum for cont}. By differentiating with respect to $z$ and taking $z=0$ we get
\begin{multline*}
 -\log\mathrm{det}(\Delta^{\sharp}_{\tau_m,\rho} + s(s-1))=\int_{0}^{\infty} e^{-t(s-1/2)^2}\left(\Theta^{\sharp}_{\tau_m,\rho}(t) - \frac{\alpha_0}{t} - \alpha_1\right)\frac{dt}{t}\\ +2\alpha_0(s-1/2)^2\log(s-1/2)-\alpha_0(s-1/2)^2 -2\alpha_1\log(s-1/2).
\end{multline*}
Trivially,
$$
\int_{0}^{\infty} e^{-t(s-1/2)^2}\left(\Theta^{\sharp}_{\tau_m,\rho}(t) - \frac{\alpha_0}{t} - \alpha_1\right)\frac{dt}{t}=o(1),
$$
as $s\to\infty$ which proves the second claim.
\end{proof}

\section{Preliminary computations}

In this section, we conduct preliminary computations necessary to deduce a relation between the regularized determinant of the twisted Laplacian and the twisted Selberg zeta function in Section \ref{sec: with main thm}. The key point will be trace formula \eqref{tracefinal} combined with Lemma \ref{lemma cont of det}.


We substitute the (R.H.S) of \eqref{tracefinal} in the integral \eqref{Lmtr} and get
\begin{align}\label{Lmtrcon}\notag
\LM\big[\Tr(e^{-tA^{\sharp}_{\tau_m,\rho}})\big]&((s-1/2)^2;z) \\&=  \frac{1}{\Gamma(z)}\int_{0}^{\infty} e^{-t(s-1/2)^2} t^{z-1}\big(I_{\tau_m,\rho}(t)+H_{\tau_m,\rho}(t)+E_{\tau_m,\rho}(t)\big)dt.
\end{align}
The term $I_{\tau_m,\rho}(t)$ is the \textit{identity contribution} in the geometrical part of the trace formula \eqref{tracefinal} given by
\begin{align}\nonumber
		& I_{\tau_m,\rho}(t)=\frac{\Vol(X)\dim(V_\rho)}{4\pi}\Bigg[\int_\R e^{-t\lambda^2}\frac{\lambda\sinh(2\pi\lambda)}{\cosh(2\pi\lambda)+\cos(\pi m)}\,d\lambda\\
		& \hspace{7cm}+ \sum_{\substack{1\leq\ell<|m|\\
\ell\textup{ odd}}}(|m|-\ell)e^{(\frac{|m|-\ell}{2})^2t}\Bigg].\label{eq. Identity cont}
\end{align}
The term $H_{\tau_m,\rho}(t)$ is the \textit{hyperbolic contribution} given by
\begin{equation}\label{hypc}
H_{\tau_m,\rho}(t)=\frac{1}{2\sqrt{4\pi t}}\sum_{[\gamma]\neq e} \frac{l(\gamma)\Tr(\rho(\gamma))}{n_{\Gamma}(\gamma)\sinh(l(\gamma)/2)}e^{-\frac{l(\gamma)^{2}}{4t}}.
\end{equation}
The term $E_{\tau_m,\rho}(t)$ is the \textit{elliptic contribution} given by
\begin{align} \nonumber
&E_{\tau_m,\rho}(t)=\sum_{[\gamma]\textup{ ell.}} \frac{\Tr\rho(\gamma)}{4M(\gamma)\sin(\theta(\gamma))}\Bigg[\int_\R e^{-t\lambda^2}\frac{\cosh(2(\pi-\theta(\gamma))\lambda)+e^{i\pi m}\cosh(2\theta(\gamma)\lambda)}{\cosh(2\pi\lambda)+\cos(\pi m)}\,d\lambda \\
		& \hspace{4cm}+2i\sign(m)\sum_{\substack{1\leq\ell<|m|\\ \ell\textup{ odd}}}e^{i\sign(m)(|m|-\ell)\theta(\gamma)}e^{(\frac{|m|-\ell}{2})^2t}\Bigg].\label{eq. elliptic cont}
\end{align}

Let us define
\begin{equation}\label{intofhyp}
H_{\tau_m,\rho}(z;s):=\frac{1}{\Gamma(z)}\int_{0}^{\infty} e^{-t(s-1/2)^2} t^{z-1}H_{\tau_m,\rho}(t)dt,
\end{equation}
\begin{equation}\label{intofid}
 I_{\tau_m,\rho}(z;s):= \frac{1}{\Gamma(z)}\int_{0}^{\infty} e^{-t(s-1/2)^2} t^{z-1}I_{\tau_m,\rho}(t)dt,
\end{equation}
and
\begin{equation}\label{intofell}
E_{\tau_m,\rho}(z;s):=\frac{1}{\Gamma(z)}\int_{0}^{\infty} e^{-t(s-1/2)^2} t^{z-1}E_{\tau_m,\rho}(t)dt
\end{equation}
\vskip.10in

We deal first with the function $H_{\tau_m,\rho}(z;s)$.
We want to compute the Laplace-Mellin transform of \eqref{hypc} which is the following integral
\begin{equation}\label{dethyp}
H_{\tau_m,\rho}(z;s)=
\frac{1}{\Gamma(z)}\int_{0}^{\infty} \frac{1}{2\sqrt{4\pi}}\sum_{[\gamma]\neq e} \frac{l(\gamma)\Tr(\rho(\gamma))}{n_{\Gamma}(\gamma)\sinh(l(\gamma)/2)}
e^{-t(s-1/2)^2} t^{z-\frac{3}{2}} e^{-\frac{l(\gamma)^{2}}{4t}}dt.
\end{equation}

\begin{lem}
The sum and the integral in \eqref{dethyp} can be interchanged for all $s$ and $z$ such that $\RE((s-1/2)^2)\gg 0$.
\end{lem}

\begin{proof}
It suffices to prove that
\begin{equation}\label{eq:hyp-series}
\sum_{[\gamma]\neq e} \left|\frac{l(\gamma)\Tr(\rho(\gamma))}{n_{\Gamma}(\gamma)\sinh(l(\gamma)/2)}
\right|\int_{0}^{\infty} e^{-t\RE((s-1/2)^2)} t^{\RE(z)-\frac{3}{2}} e^{-\frac{l(\gamma)^{2}}{4t}}dt<\infty.
\end{equation}
In order to compute the integral in\eqref{eq:hyp-series}, we recall \cite{GR07}, formula 3.478.4 with $p=1$, which
states that
\begin{equation}\label{eq. basic gr identity}
\int_{0}^{\infty} t^{\nu-1} e^{-\beta t-\frac{\gamma}{t}}dt=2\left(\frac{\gamma}{\beta}\right)^{\nu/2}K_\nu(2\sqrt{\beta\gamma}),
\end{equation}
for $\RE(\beta)>0$ and $\RE(\gamma)>0$.
In this formula, the square-root is principal and $K_\nu$ is the $K$-Bessel function.
In our setting, we then have that
\begin{align*}
\int_{0}^{\infty} e^{-t\RE((s-1/2)^2)}& t^{\RE(z)-\frac{3}{2}} e^{-\frac{l(\gamma)^{2}}{4t}}dt\\&= 2\left(\frac{l(\gamma)}{2\sqrt{\RE(s-1/2)^2}}\right)^{\RE(z)-1/2}K_{\RE(z)-1/2}
\left(l(\gamma)\sqrt{\RE(s-1/2)^2}\right).
\end{align*}
The $K$-Bessel function decays exponentially.  More specifically, from \cite{GR07}, 8.451.6 we have that
$$
K_{\RE(z)-1/2} \left(l(\gamma)\sqrt{\RE(s-1/2)^2}\right)
\ll (l(\gamma))^{-1/2} e^{-l(\gamma)\sqrt{\RE(s-1/2)^2}}
$$
for any fixed complex $z$.  Therefore, for $\RE((s-1/2)^2)\gg 0$ we get that
\begin{align*}
\sum_{[\gamma]\neq e} &\left|\frac{l(\gamma)\Tr(\rho(\gamma))}{n_{\Gamma}(\gamma)\sinh(l(\gamma)/2)}
\right|\int_{0}^{\infty} e^{-t\RE((s-1/2)^2)} t^{\RE(z)-\frac{3}{2}} e^{-\frac{l(\gamma)^{2}}{4t}}dt\\& \ll
\sum_{[\gamma]\neq e} \frac{|\Tr(\rho(\gamma))|}{n_{\Gamma}(\gamma)\sinh(l(\gamma)/2)}
(l(\gamma))^{\RE(z)} e^{-l(\gamma)\sqrt{\RE(s-1/2)^2}}
\end{align*}
Equation (4.4) from \cite{frahm2023twisted} asserts the existence of constants $c',c$ independent of $\gamma$,
such that
\begin{equation}\label{eq. trace bound}
|\Tr(\rho(\gamma))|\leq c' e^{cl(\gamma)}.
\end{equation}
With all this, one now can apply standard bounds for the growth of the length spectrum $\{l(\gamma)\}$ to
conclude that \eqref{eq:hyp-series} converges for
$\RE(s-1/2)^2\gg 0$ and any $z\in\mathbb{C}$, so then it is allowed to interchange the sum and the integral in \eqref{dethyp}.
\end{proof}

\begin{lem}
There exists a constant $r>0$ such that for all $s\in\C$ with $\RE(s-1/2)\geq r$, the function $H_{\tau_m,\rho}(z;s)$ is holomorphic in
$z$-variable on the whole complex plane.
\end{lem}

\begin{proof}
 Let $C>c+1$, where $c$ is the constant from \eqref{eq. trace bound}. Then, the series
  $$
  \sum_{[\gamma]\neq e} \frac{|\Tr(\rho(\gamma))|}{n_{\Gamma}(\gamma)\sinh(l(\gamma)/2)}
(l(\gamma))^{\RE(z)} e^{-C l(\gamma)}
  $$
converges uniformly and absolutely on every compact subset of the complex $z$-plane hence so does the integral in the
right-hand-side of \eqref{dethyp}. Given that $1/\Gamma(z)$ is a holomorphic function of $z$, the assertion is proved that for all
$z\in \C$ and $s\in\C$ with $\RE(s-1/2)^2\geq C'$ for some positive constant $C'$.
\end{proof}

\begin{lem}
 For all $s\in\C$ with $\RE(s-1/2)^2\geq C$ and all $z\in\C$, we have that
\begin{equation}\label{H in terms of K Bessel}
H_{\tau_m,\rho}(z;s)=
\frac{1}{\Gamma(z)}\frac{1}{\sqrt{4\pi}}\sum_{[\gamma]\neq e} \frac{l(\gamma)\Tr(\rho(\gamma))}{n_{\Gamma}(\gamma)\sinh(l(\gamma)/2)}
\left(\frac{l(\gamma)}{2s-1}\right)^{z-1/2}K_{z-1/2}((s-1/2)l(\gamma)).
\end{equation}
\end{lem}

\begin{proof}
The stated formula follows by interchanging the sum and the integral in \eqref{dethyp} and by applying  \eqref{eq. basic gr identity} with $\nu=z-1/2$, $\beta=(s-1/2)^2$ and $\gamma=l(\gamma)^2/4$.
\end{proof}

\begin{prop}\label{hypfinal}
For all $s\in\C$,
\begin{equation}\label{eq. deriv at z=0}
\frac{d}{dz}\bigg\vert_{z=0}H_{\tau_m,\rho}(z;s)=-\log Z(s;\rho)
\end{equation}
\end{prop}

\begin{proof}
  Assume first that $s\in\C$ with $\RE(s-1/2)^2\geq C$. Upon
  differentiating \eqref{H in terms of K Bessel} with respect to $z$ and taking $z=0$ we get that
  $$
  \frac{d}{dz}\bigg\vert_{z=0}H_{\tau_m,\rho}(z;s)=\frac{1}{\sqrt{4\pi}}\sum_{[\gamma]\neq e} \frac{l(\gamma)\Tr(\rho(\gamma))}{n_{\Gamma}(\gamma)\sinh(l(\gamma)/2)}
\left(\frac{l(\gamma)}{2s-1}\right)^{-1/2}K_{-1/2}((s-1/2)l(\gamma)).
  $$
Using the evaluation for $K_{-1/2}(x)$  from \cite{GR07}, formula 8.649.3, we get
$$
 \frac{d}{dz}\bigg\vert_{z=0}H_{\tau_m,\rho}(z;s)=\sum_{[\gamma]\neq e} \frac{\Tr(\rho(\gamma))}{2n_{\Gamma}(\gamma)\sinh(l(\gamma)/2)}
e^{-(s-1/2)l(\gamma)}=-\log Z(s;\rho),
$$
where the last equality follows from \cite{BFS23}, section 4. The claim now follows by the principle of analytic continuation.
\end{proof}

\vskip.10in

Let us now consider the functions $I_{\tau_m,\rho}(z;s)$ and $E_{\tau_m,\rho}(z;s)$ as defined in \eqref{intofid} and \eqref{intofell}, respectively.

\begin{prop}\label{idfinal}
For $\RE(s)\gg0$, the functions  $I_{\tau_m,\rho}(z;s)$ and $E_{\tau_m,\rho}(z;s)$
can be meromorphically continued to entire $z$-plane, and the continuations are holomorphic at $z=0$.
\end{prop}

\begin{proof} Let
\begin{equation}\label{eq. c rho defn}
C_\rho(X):=\frac{\mathrm{dim}(V_\rho)\mathrm{vol}(X)}{4\pi}.
\end{equation}
  We start by observing that
  $$
 2 \int_0^\infty \lambda e^{-\lambda^2t}d\lambda=\frac{1}{t}.
  $$
By subtracting and adding the term $1/t$, we get that
$$
 I_{\tau_m,\rho}(t)=C_\rho(X)\left(\int_0^\infty  \frac{-4\lambda(1+e^{2\pi\lambda}\cos(\pi m))}{e^{4\pi\lambda} + 2e^{2\pi\lambda}\cos(\pi m) +1}  e^{-\lambda^2t}d\lambda +\frac{1}{t}+ \sum_{\substack{1\leq\ell<|m|\\
\ell\textup{ odd}}}(|m|-\ell)e^{(\frac{|m|-\ell}{2})^2t}\right).
 $$
Let
$$
f(m,\lambda):= \frac{4(1+e^{2\pi\lambda}\cos(\pi m))}{e^{4\pi\lambda} + 2e^{2\pi\lambda}\cos(\pi m) +1}= \frac{2}{1+e^{2\pi(\lambda+im/2)}}+\frac{2}{1+e^{2\pi(\lambda-im/2)}}.
$$
By the rapid decay of $f(m,\lambda)$ in $\lambda$,  we can apply Fubini-Tonelli theorem for $\RE(z)>1$  to deduce that
\begin{multline*}
I_{\tau_m,\rho}(z;s)=\frac{C_\rho(X)}{\Gamma(z)} \left( -\int_{0}^{\infty} \bigg(\int_{0}^{\infty} e^{-t((s-1/2)^2+\lambda^2)} t^{z-1} dt\bigg)\lambda f(m,\lambda) d\lambda +\right.\\
+ \left.\int_{0}^{\infty} e^{-t(s-1/2)^2} t^{z-2} dt + \sum_{\substack{1\leq\ell<|m|\\
\ell\textup{ odd}}}(|m|-\ell)  \int_{0}^{\infty} e^{-t((s-1/2)^2-(|m|-\ell)^2/4)} t^{z-1} dt  \right).
\end{multline*}
With the assumption that, $\RE(z)>1$ we can evaluate the above integrals in variable $t$. Specifically, we have that
$$
\int_{0}^{\infty} e^{-t((s-1/2)^2+\lambda^2)} t^{z-1} dt=\frac{\Gamma(z)}{((s-1/2)^2+\lambda^2)^z}
$$
as well as
$$
\int_{0}^{\infty} e^{-t(s-1/2)^2} t^{z-2} dt=\frac{\Gamma(z-1)}{(s-1/2)^{2(z-1)}}
$$
and
$$
\int_{0}^{\infty} e^{-t((s-1/2)^2-(|m|-\ell)^2/4)} t^{z-1} dt=\frac{\Gamma(z)}{((s-1/2)^2- (|m|-\ell)^2/4)^z}.
$$
Therefore,
\begin{align}\label{idfinalz}\notag
I_{\tau_m,\rho}(z;s)=C_\rho(X)\bigg(-\int_{0}^{\infty}&\frac{\lambda f(m,\lambda)d\lambda}{((s-1/2)^2+\lambda^2)^z} + \frac{(s-1/2)^{2-2z}}{(z-1)}
\\ & +\sum_{\substack{1\leq\ell<|m|\\
\ell\textup{ odd}}}\frac{(|m|-\ell) }{((s-1/2)^2- (|m|-\ell)^2/4)^z} \bigg).
\end{align}
The integral over $\lambda$ in the above display converges uniformly on every compact subset of the complex $z$-plane; hence,
it is analytic function of $z$ at $z=0$. Trivially, the second term and the third term in the above display are also holomorphic
at $z=0$. This proves the claim for $I_{\tau_m,\rho}(z;s)$.

The proof for $E_{\tau_m,\rho}(z;s)$ is analogous. We start by observing that $0<\theta(\gamma)<\pi$; hence, for $\RE(z)>1$, due to exponential decay of the integrand in the first term defining the elliptic contribution, it is possible to apply Fubini-Tonelli theorem and deduce that
\begin{align} \nonumber
E_{\tau_m,\rho}(z;s)=\sum_{[\gamma]\textup{ ell.}} \frac{\Tr\rho(\gamma)}{4M(\gamma)\sin(\theta(\gamma))}&\Bigg[\int_\R \frac{\cosh(2(\pi-\theta(\gamma))\lambda)+e^{i\pi m}\cosh(2\theta(\gamma)\lambda)}{(\cosh(2\pi\lambda)+\cos(\pi m))((s-1/2)^2 + \lambda^2)^z}\,d\lambda \\
		& +2i\sign(m)\sum_{\substack{1\leq\ell<|m|\\ \ell\textup{ odd}}}\frac{e^{i\sign(m)(|m|-\ell)\theta(\gamma)}}{((s-1/2)^2- (|m|-\ell)^2/4)^z} \Bigg].\label{eq. elliptic cont final}
\end{align}
Since $0<\theta(\gamma)<\pi$, the integral in the above display converges uniformly in any compact subset of the complex $z$-plane,
thus defines a holomorphic function at $z=0$. This completes the proof.
\end{proof}

\section{Expressing the determinant in terms of the twisted Selberg zeta function}\label{sec: with main thm}

We can now proceed to the proof of the main theorem of this article.

\begin{thm}\label{thm:main_theorem}
With the notation as above, the twisted determinant of the Laplacian $\mathrm{det}(\Delta_{\tau_m,\rho}^{\sharp} +s(s-1))$
and the twisted Selberg zeta function $Z(s;\rho)$ satisfy the relation
\begin{equation}\label{eq. det final}
\mathrm{det}(\Delta_{\tau_m,\rho}^{\sharp} +s(s-1))= Z(s;\rho)Z_I(s,\rho)Z_{\rm ell}(s;\rho)e^{\widetilde C},
\end{equation}
where the functions $Z_{I}(s;\rho)$ and $Z_{\rm ell}(s;\rho)$ are given by
\begin{align}\label{e:ZI_def}\notag
Z_{I}(s;\rho) = &\exp\bigg\{2C_\rho(X)\bigg(s\log(2\pi) + s(1-s) + \left(\frac{1+m}{2}\right) \log\Gamma\left(s+m/2\right)
\\ &+ \left(\frac{1-m}{2}\right) \log \Gamma\left(s-m/2\right) - \log G\left(s+m/2+1\right)-\log G\left(s-m/2+1\right)\bigg)\bigg\},
\end{align}
and
\begin{equation}\label{eq. Z ell final}
Z_{\rm ell}(s;\rho) = \prod_{j=1}^r \prod_{\ell=0}^{\nu_j-1}\bigg(\Gamma\left(\frac{s-\tfrac{m}{2}+\ell}{\nu_j}\right)^{C_m(j,\ell;\rho)} \Gamma\left(\frac{s+\tfrac{m}{2}+\ell}{\nu_j}\right)^{\tilde C_m(j,\ell;\rho) }\bigg)
\end{equation}
where the constant $C_\rho(X)$ is defined by \eqref{eq. c rho defn}, $C_m(j,\ell;\rho)$ and $\tilde C_m(j,\ell;\rho)$ are rational numbers defined by \eqref{eq. constant C_m def} and \eqref{eq. constant C_m tilde def}, and with the constant $\widetilde{C}$ given by \eqref{eq. tilde C defn}.
\end{thm}


\begin{proof}

Assume $\RE(s)\gg 0$. By combining \eqref{Lmtrcon}, \eqref{intofhyp}, \eqref{intofid}, \eqref{intofell} Propositions \ref{hypfinal}
and \ref{idfinal}, we have
\begin{align}\notag
\frac{d}{dz}\LM\big[\Tr(e^{-tA^{\sharp}_{\tau_m,\rho}})\big]((s-1/2)^2;z)\bigg\vert_{z=0}&=-\log \mathrm{det}(\Delta_{\tau_m,\rho}^{\sharp} +s(s-1))\\&= -\log Z(s;\rho) + \frac{d}{dz}\left(I_{\tau_m,\rho}(z;s)+ E_{\tau_m,\rho}(z;s)\right)\bigg\vert_{z=0} .\label{eq. det pre-final}
\end{align}
Let us use \eqref{idfinalz} and compute the derivative of $I_{\tau_m,\rho}(z;s)$ at $z=0$.
Specifically, we have that
\begin{align*}
\frac{d}{dz}I_{\tau_m,\rho}(z;s)\bigg\vert_{z=0}&= C_\rho(X)\bigg[ (s-1/2)^2(2\log(s-1/2) -1) \\&+ \int_0^\infty \lambda f(m,\lambda) \log((s-1/2)^2+\lambda^2)d\lambda\\&-  \sum_{\substack{1\leq\ell<|m|\\ \ell\textup{ odd}}}(|m|-\ell)  \log((s-1/2)^2- (|m|-\ell)^2/4)\bigg].
\end{align*}
The above functions are holomorphic in $s$ for $\RE(s)>1$.  As such, we can differentiate the above identity with respect to $s$ to deduce that
\begin{multline*}
\frac{1}{2s-1}\frac{d}{ds}\left(\frac{d}{dz}I_{\tau_m,\rho}(z;s)\bigg\vert_{z=0}\right) \\= C_\rho(X)\left(\log(s-1/2)^2 + \int_0^{\infty}\frac{\lambda f(m,\lambda) d\lambda}{(\lambda^2 + (s-1/2)^2)} -  \sum_{\substack{1\leq\ell<|m|\\ \ell\textup{ odd}}} \frac{(|m|-\ell) }{((s-1/2)^2- (|m|-\ell)^2/4)}\right).
\end{multline*}
We claim that
\begin{align}\notag
\log(s-1/2)^2 + \int_0^{\infty}&\frac{\lambda f(m,\lambda) d\lambda}{(\lambda^2 + (s-1/2)^2)} -  \sum_{\substack{1\leq\ell<|m|\\ \ell\textup{ odd}}} \frac{(|m|-\ell) }{((s-1/2)^2- (|m|-\ell)^2/4)} \\&
= \psi\left(s+\frac{m}{2}\right) + \psi\left(s-\frac{m}{2}\right),\label{eq:digamma}
\end{align}
where $\psi$ denotes the digamma function.  Recall that $\psi(s)= \log s +O(s^{-1})$ as $s\to\infty$.
With this, it suffices to prove the claim by showing that the derivatives on the two sides of \eqref{eq:digamma} are equal.
In other words, the claim reduces to the assertion that
\begin{align*}
\frac{1}{(s-1/2)^2} &+\sum_{\substack{1\leq\ell<|m|\\ \ell\textup{ odd}}} \frac{(|m|-\ell) }{((s-1/2)^2- (|m|-\ell)^2/4)^2} - \int_0^{\infty}\frac{\lambda f(m,\lambda) d\lambda}{(\lambda^2 + (s-1/2)^2)^2}\\& =\frac{1}{2s-1}\left[\psi'\left(s+\frac{m}{2}\right) + \psi'\left(s-\frac{m}{2}\right)\right].
\end{align*}

From the identity
$$
\frac{1}{(s-1/2)^2}=\int_0^{\infty}\frac{2\lambda d\lambda}{(\lambda^2 + (s-1/2)^2)^2}
$$
we have that
\begin{align*}
\frac{1}{(s-1/2)^2} &- \int_0^{\infty}\frac{\lambda f(m,\lambda) d\lambda}{(\lambda^2 + (s-1/2)^2)^2}\\&=
\frac{-1}{2i(2s-1)} \int_0^{\infty}\left(\frac{1}{(\lambda+i(s-1/2))^2} -  \frac{1}{(\lambda-i(s-1/2))^2}\right)(2-f(m,\lambda))d\lambda\\&=
\frac{-1}{2i(2s-1)} (I^+(s,m) - I^-(s,m)),
\end{align*}
where, just to be clear,
$$
I^{\pm}(s,m) =  \int_0^{\infty}F^{\pm}(\lambda)d\lambda
\,\,\,\,\,
\textrm{\rm with}
\,\,\,\,\,
F^{\pm}(z)= \frac{2-f(m,z)}{(z\pm i(s-1/2))^2}
$$
In order to compute $I^+(s,m)$, consider the integral of
$F^+(z)$
along the contour $C^+(R)$ which consists of the real line from $0$ to $R$, then along
the  circle $|z|=R$ with $\arg z$ ranging from $0$ to $\pi/2$, and then along the imaginary line segment $\arg z=\pi/2$ from $iR$ to $0$.
Actually, this is a regularized integral since the function
\begin{equation}\label{eq. 2-f}
2-f(m,z)=2\left[1-\frac{1}{1+e^{2\pi(\lambda+im/2)}}-\frac{1}{1+e^{2\pi(\lambda-im/2)}}\right]
\end{equation}
possesses simple poles at $\lambda=i(1/2+k\pm m/2)$ for $k\in\mathbb{Z}$ each with residue equal to $1/\pi$.
The regularization consists of the standard principal value (PV) method, namely that the
the integration path along the imaginary axes is deformed by replacing
the part of the imaginary axes at a very small distance $\varepsilon>0$ from the pole by half-circles of radius $\varepsilon>0$
centered at the pole and belonging to the first quadrant.
After letting $\varepsilon\to 0$ the integral along the half circle tends to the residue at this simple
pole multiplied by $-\pi i$, where the minus sign reflects the orientation of the path of integration. If $m=1$, the pole at $\lambda=0$ will be multiplied by $-i\pi/2$. However, in this case $\lambda=0$ will appear as pole in two terms appearing on the right-hand side of \eqref{eq. 2-f}. By following this approach, we get
for any $R>0$ such that $R\neq 1/2+k\pm m/2$ for some $k\in\mathbb{Z}$ that
\begin{align*}
&\int\limits_{C^+(R)}F^+(z)dz= \int\limits_0^R F^+(z)dz + \int\limits_{|z|=R; \, \arg z\in(0,\pi/2)} F^+(z)dz + \frac{i\delta_{m=1}}{(s-1/2)^2}\\ - &\pi i \left(\sum_{k\in\mathbb{Z}:(1/2+k-m/2)\in(0, R)} \frac{-1}{\pi(k+s-m/2)^2}+ \sum_{k\in\mathbb{Z} :(1/2+k+m/2)\in(0, R)} \frac{-1}{\pi(k+s+m/2)^2} \right),
\end{align*}
where $\delta_{m=1}$ equals $1$ if $m=1$ and equals zero, otherwise.

After letting $R\to \infty$ we conclude that
\begin{align*}
I^{+}(s,m) = \int_0^{\infty}F^{+}(z)dz=\pi i &\left(\sum_{k\in\mathbb{Z}:(1/2+k-m/2)\geq 0}
\frac{-1}{\pi(k+s-m/2)^2} \right.\\& \left.+  \sum_{k\in\mathbb{Z} :(1/2+k+m/2)> 0} \frac{-1}{\pi(k+s+m/2)^2} \right)
\end{align*}
Analogously, we consider the integral of the function
$$
F^-(z)= \frac{2-f(m,z)}{(z-i(s-1/2))^2}
$$
in negative direction along the contour $C^-$ which is defined as the reflection of the contour $C^+$ with respect to the $x$-axes,
also in the regularized sense as explained above.  Upon doing so, we conclude that
\begin{align*}
I^{-}(s,m)=\int_0^{\infty}F^{-}(z)dz=- \pi i &\left(\sum_{k\in\mathbb{Z}:(-1/2-k-m/2)< 0} \frac{-1}{\pi(k+s+m/2)^2} +\right. \\& \left.+  \sum_{k\in\mathbb{Z} :(-1/2-k+m/2)\leq 0} \frac{-1}{\pi(k+s-m/2)^2} \right).
\end{align*}
Therefore
\begin{align*}
&\frac{-1}{2i(2s-1)}(I^+(s,m) - I^-(s,m))\\&=
\frac{1}{(2s-1)}\left(\sum_{k\in\mathbb{Z}:(1/2+k-m/2)\geq 0} \frac{1}{(k+s-m/2)^2} +  \sum_{k\in\mathbb{Z} :(1/2+k+m/2)> 0} \frac{1}{(k+s+m/2)^2} \right),
\end{align*}
hence
\begin{align*}
\frac{-1}{2i(2s-1)} (I^+(s,m) - I^-(s,m)) & +  \sum_{\substack{1\leq\ell<|m|\\ \ell\textup{ odd}}} \frac{(|m|-\ell) }{((s-1/2)^2- (|m|-\ell)^2/4)^2}
\\&=  \frac{1}{(2s-1)}\left(\sum_{k= 0}^{\infty} \frac{1}{(k+s-m/2)^2} +  \sum_{k= 0}^{\infty} \frac{1}{(k+s+m/2)^2} \right)\\
&=\frac{1}{2s-1}\left[\psi'\left(s+\frac{m}{2}\right) + \psi'\left(s-\frac{m}{2}\right)\right].
\end{align*}
In summary, we have proved that
\begin{equation*}
\frac{d}{ds}\left(\frac{d}{dz}I_{\tau_m,\rho}(z;s)\bigg\vert_{z=0}\right)= C_\rho(X)(2s-1)\left(\psi\left(s+\frac{m}{2}\right) + \psi\left(s-\frac{m}{2}\right)\right).
\end{equation*}
On the other hand, from \cite{FiBook87}, Remark 3.1.3 on p. 114 with $k=m/2$, we see that\footnote{The identity \eqref{eq. identity term pre-comp} follows from \eqref{log der barnes gamma}}
\begin{equation} \label{eq. identity term pre-comp}
\frac{d}{ds}\left(\frac{d}{dz}I_{\tau_m,\rho}(z;s)\bigg\vert_{z=0}\right)=-\frac{d}{ds}
\log Z_{I}(s;\rho)
\end{equation}
which is the expression needed to evaluate the identity contribution in \eqref{eq. det pre-final}.

Let us now deal with the elliptic contribution in \eqref{eq. det pre-final}. To begin, note that
 \begin{align}\label{eq. ell cont beginning} \notag
-\frac{1}{2s-1}\frac{d}{ds}& \left(\frac{d}{dz}E_{\tau_m,\rho}(z;s)\bigg\vert_{z=0}\right) \\&= \notag \sum_{[\gamma]\textup{ ell.}}\frac{\Tr\rho(\gamma)}{4M(\gamma)\sin(\theta(\gamma))}\times \Bigg[\int_\R \frac{\cosh(2(\pi-\theta(\gamma))\lambda)+e^{i\pi m}\cosh(2\theta(\gamma)\lambda)}{(\cosh(2\pi\lambda)+\cos(\pi m))((s-1/2)^2 + \lambda^2)}\,d\lambda \\& \hspace{4mm}
		+2i\sign(m)\sum_{\substack{1\leq\ell<|m|\\ \ell\textup{ odd}}}\frac{e^{i\sign(m)(|m|-\ell)\theta(\gamma)}}{((s-1/2)^2- (|m|-\ell)^2/4)} \Bigg].
\end{align}
By applying Lemma 2.7 from p. 432 of \cite{He} for $\RE(s)\gg 0$ we conclude that
 \begin{multline*}
-\frac{1}{2s-1}\frac{d}{ds}\left(\frac{d}{dz}E_{\tau_m,\rho}(z;s)\bigg\vert_{z=0}\right)= \sum_{[\gamma]\textup{ ell.}}\frac{\Tr\rho(\gamma)}{4M(\gamma)\sin(\theta(\gamma))} \frac{i}{s-1/2}\times \\ \sum_{j=0}^{\infty}\left(\frac{e^{-2i\theta(\gamma)(j-m/2+1/2)}}{j-m/2+s} - \frac{e^{2i\theta(\gamma)(j+m/2+1/2)}}{j+m/2+s}\right).
\end{multline*}
Let us now further simplify the above expression. To do so, enumerate the inconjugate elliptic elements by $\gamma_1,\ldots, \gamma_r$.
Then $\theta(\gamma_k)=\pi l/\nu_j$ for some $l\in\{1,\ldots,\nu_j-1\}$. From p. 58--59 of  \cite{FiBook87}  we have that
\eqref{eq. ell cont beginning} is equal to
 \begin{align*}
 \sum_{j=1}^r \frac{\Tr\rho(\gamma_j)}{2\nu_j\sin(\theta(\gamma_j))} \frac{ie^{i\theta(\gamma_j)m}}{2s-1}\frac{1}{\nu_j} \sum_{\ell=0}^{\nu_j-1}\bigg(&e^{i\theta(\gamma_j)(2\ell+1)}\psi\left(\frac{s+m/2+1}{\nu_j}\right)
\\&-
e^{-i\theta(\gamma_j)(2\ell+1)}\psi\left(\frac{s-m/2+1}{\nu_j}\bigg)
\right).
\end{align*}
The expression on the right-hand side of the above equation coincides with the right-hand side of the last display on
p. 61 of \cite{FiBook87} with $2k=m$, and the elliptic conjugacy clases denoted by $R$ instead of $\gamma$. Therefore, $\frac{-1}{2s-1}\frac{d}{ds}\left(\frac{d}{dz}E_{\tau_m,\rho}(z;s)\bigg\vert_{z=0}\right)$ equals the integral over the fundamental domain of the elliptic contribution to the Green's function as studied in \cite{FiBook87}.  In particular, from Corollary 2.3.5 on page 68 of \cite{FiBook87} we have that
\begin{equation}\label{eq. ell cont final}
-\frac{1}{2s-1}\frac{d}{ds}\left(\frac{d}{dz}E_{\tau_m,\rho}(z;s)\bigg\vert_{z=0}\right)= -\frac{1}{2s-1}\frac{Z'_{\rm ell,0}(s;\rho)}{Z_{\rm ell,0}(s;\rho)}(s),
\end{equation}
where
\begin{align}\label{e:Zell_def}\notag
Z_{\rm ell,0}(s;\rho) = \prod_{j=1}^r &\bigg\{\nu_j^{\dim(V_\rho)(1-\frac{1}{\nu_j})s} \left(\Gamma\left(s-\tfrac{m}{2}\right)\Gamma\left(s+\tfrac{m}{2}\right)\right)^{-\frac{1}{2}\dim(V_\rho) (1-\frac{1}{\nu_j})}
\\ &\times \prod_{\ell=0}^{\nu_j-1} \bigg(\Gamma\left(\frac{s-\tfrac{m}{2}+\ell}{\nu_j}\right)^{\frac{\alpha_j(\ell)}{\nu_j} } \Gamma\left(\frac{s+\tfrac{m}{2}+\ell}{\nu_j}\right)^{\frac{\tilde{\alpha}_j(\ell)}{\nu_j} }\bigg)\bigg\}.
\end{align}

Next, we will transform \eqref{e:Zell_def} by using \eqref{eq. alpha l rep} nd \eqref{eq. alpha tilde l rep}. First, by applying the product formula for the Gamma function (see, e.g. \cite{GR07}, formula 8.335) one has that
\begin{align*}
\nu_j^{\frac{1}{2}\dim(V_\rho)(1-\frac{1}{\nu_j})s}& \Gamma\left(s\pm\tfrac{m}{2}\right)^{-\frac{1}{2}\dim(V_\rho) (1-\frac{1}{\nu_j})}
\prod_{\ell=0}^{\nu_j-1} \Gamma\left(\frac{s\pm\tfrac{m}{2}+\ell}{\nu_j}\right)^{\frac{1}{2}\dim(V_\rho) (1-\frac{1}{\nu_j}) } \\ &= \left((2\pi)^{\frac{\nu_j-1}{2}}\nu_j^{1/2\mp m/2}\right)^{\frac{\dim(V_\rho)(\nu_j-1)}{2\nu_j}}.
\end{align*}
From \eqref{e:Zell_def}, in view of definitions \eqref{eq. alpha l rep} and \eqref{eq. alpha tilde l rep} now we deduce that
$$
Z_{\rm ell,0}(s;\rho) = Z_{\rm ell}(s;\rho) \prod_{j=1}^r\left((2\pi)^{\nu_j-1}\nu_j\right)^{\frac{\dim(V_\rho)(\nu_j-1)}{2\nu_j}}.
$$

With all this, upon combining with \eqref{eq. det pre-final}, \eqref{eq. identity term pre-comp} and  \eqref{eq. ell cont final}
we arrive at the formula that
\begin{equation}\label{eq:main_formula_with_constant}
\log\mathrm{det}(\Delta_{\tau_m,\rho}^{\sharp} +s(s-1))=\log Z(s;\rho)+\log Z_I(s;\rho) + \log Z_{\rm ell}(s;\rho) + \widetilde{C}_1,
\end{equation}
for some constant $\widetilde{C}_1$. The constant $\widetilde{C}_1$ can be evaluated by determining the asymptotic behavior
of \eqref{eq:main_formula_with_constant} for real $s$ tending to infinity.  The asymptotic behavior of the determinant out to $o(1)$
comes from Lemma \ref{lemma cont of det}. Therefore, the constant $\widetilde C_1$ is to be determined so that the constant term in the asymptotic expansion as $s\to\infty$ of the right-hand side of \eqref{eq:main_formula_with_constant} equals zero.

The formulas for $Z_{I}$ and $Z_{\rm ell}$, namely \eqref{e:ZI_def}
and \eqref{eq. Z ell final}, are readily available to determine their asymptotic expansions out to $o(1)$.  To do so for the identity contribution, one needs \eqref{asmBarnes} and the expansion from \cite[8.344]{GR07} of the function $\log\Gamma(s)$. The constant term in the asymptotic expansion of the term \eqref{eq. Z ell final} is easily computed by combining the expansion from \cite[8.344]{GR07} of the function $\log\Gamma(s)$ with the identity \eqref{eq. sum of Cm s} and the identity
$$
\sum_{\ell=0}^{\nu_j-1}\ell\sin\left(\frac{\pi k}{\nu_j}(2\ell+1)\right)=-\frac{\nu_j}{2\sin\left(\frac{\pi k}{\nu_j}\right)}.
$$
It is immediate that $\log Z(s;\rho)$ is $o(1)$ for large real $s$.
When combining all of these expansions, one gets that $\widetilde{C}_1$ is given as stated
in \eqref{eq. tilde C defn}, as well as the values for $\alpha_{0}$
and $\alpha_{1}$ in Lemma \ref{lemma cont of det}.

With all this, the proof is complete.
\end{proof}

\section{Analysis of the torsion factor for some representations}

We now turn our attention to the torsion factor \eqref{eq. tilde C defn} .
The torsion factor is well defined for both cyclic\footnote{Recall that $\rho$ is cyclic if $\rho(u)=I_n$ meaning that $m=0$.} and acyclic representations $\rho$, hence we expect it describes a torsion different from the Reidemeister torsion.

\begin{ex} \label{ex:Yamaguchi} \rm
Let us compute the torsion factor associated to the even-dimensional representation $\rho_{2N}:
\pi_1(X_1) \to\mathrm{SL}_{2N}(\mathbb{C})$ that was studied by Yamaguchi in \cite{Ya22}. For such a
representation $\rho_{2N}$, we have $m=1$ and
\begin{equation}\label{eq:yamaguchi_example}
\mathrm{Tr}(\rho(\gamma_j)^k)e^{i\pi k/\nu_j}=\sum_{p=-(N-1)}^{N}e^{2\pi i p k/\nu_j}
=\sum_{p=1}^{N}e^{2\pi i p k/\nu_j} + \sum_{p=0}^{N-1}e^{-2\pi i p k/\nu_j}
\end{equation}
for $k=1,\ldots, \nu_j-1$.  If $\nu_j\mid N$ then \eqref{eq:yamaguchi_example} is zero, so let us now
assume $\nu_j$ is not a factor of $N$.  Let $r_j(N):= 2N-\nu_j\lfloor \frac{2N}{\nu_j}\rfloor$ where $\lfloor x \rfloor$ denotes the
integer part of a real number $x$. Moreover, set
\begin{multline*}
A_j(N)= \bigg\{r_j\in \{0,1,\ldots\nu_j-1\} : \\
r_j\equiv \left(N-\nu_j\lfloor \frac{2N}{\nu_j}\rfloor-q \right)({\rm mod} \, \nu_j)\,\,\, \text{for some}\,\,\, q=0,\ldots,r_j(N)-1 \bigg\}.
\end{multline*}
Then, a simple argument based on counting residue classes modulo $\nu_j$ among the elements of the
set $\{-(N-1), -(N-2),\ldots, 0, \ldots, (N-1), N)\}$ yields that
\begin{align*}
\mathrm{Tr}(\rho(\gamma_j)^k)e^{i\pi k/\nu_j} &= \left(\lfloor \frac{2N}{\nu_j}\rfloor+1\right) \sum_{r_j\in A_j(N) } e^{2\pi i r_j k/\nu_j} + \lfloor \frac{2N}{\nu_j}\rfloor \sum_{r_j\in A_j(N)^c} e^{2\pi i r_j k/\nu_j}\\ &= \sum_{r_j\in A_j(N) } e^{2\pi i r_j k/\nu_j},
\end{align*}
where $A_j(N)\cup A_j(N)^c=\{0,1,\ldots,\nu_j-1\}$.

For a fixed $r_j \in \{0,1,\ldots,\nu_j-1\}$, we use the recurrent formula from Corollary 2 of \cite{JKS24}.
Specifically, in the notation of \cite{JKS24}, we set $r_{j}=r$, $n=1$ and $m=\nu_{j}$.  Then when comparing equations (4)
and (9) of \cite{JKS24}, we have that
$$
\sum_{k=1}^{\nu_j-1}\frac{ e^{2\pi i r_j k/\nu_j}}{\sin^2\left(\frac{\pi k}{\nu_j}\right)}=2\nu_j c_{\nu_{j},r_j}(0).
$$
Now, we use equations (10), (50) and (51) of \cite{JKS24}  to get that
$$
2\nu_j c_{m,r_j}(0)=\frac{\nu_j^2-6\nu_jr_j+6r_j^2-1}{3}.
$$
Therefore, the torsion factor for the representation $\rho_{2N}$ equals
\begin{equation}\label{eq:Yam_torsion}
  \widetilde C =-2N\chi(X)\left(2\zeta'(-1) - \log\sqrt{2\pi} \right) + \sum_{j=1}^{r}\frac{\log\nu_j}{2\nu_j} \sum_{r_j\in A_j(N)} \frac{\nu_j^2-6\nu_jr_j+6r_j^2-1}{3}.
\end{equation}
It is interesting to note that the behavior of the torsion factor in
\eqref{eq:Yam_torsion}
as $N\to\infty$ is similar to that obtained
in Section 3.4 of \cite{Ya22}.  Specifically, it is trivial to observe that
\begin{equation} \label{eq. lim of tilde C}
\lim_{N\to\infty}\frac{ \widetilde C}{2N}=-\chi(X)\left(2\zeta'(-1) - \log\sqrt{2\pi} \right),
\end{equation}
which is a constant depending only upon the topological data related to $X$ namely the Euler characteristic.
\end{ex}

Indeed, \eqref{eq. lim of tilde C} holds in further generality, which we now state.

\begin{coro}
  With the notation as above consider a sequence of representations $\{\rho_{N}\}_{N\geq 1}$, such that $\mathrm{Tr}
  (\rho_{N}(\gamma_j))$ is uniformly bounded by some constant $C(X)$ which is independent of $j$ and $N$.Assume also that the sequence of dimensions $\{\dim (V_{\rho_N})\}_{N\geq 1}$ of representations $\rho_N$ is such that $\lim_{N\to\infty} \dim (V_{\rho_N})=\infty$. We
  then have that
  \begin{equation*}
\lim_{N\to\infty}\frac{ \widetilde C}{\dim (V_{\rho_N})}=-\chi(X)\left(2\zeta'(-1) - \log\sqrt{2\pi} \right).
\end{equation*}
  \end{coro}

An example of a sequence of non-unitary representations satisfying the assumption of the above
corollary is the sequence $\{\rho_{2N}\}_{N\geq 1}$ of $2N$-dimensional complex acyclic and non-unitary representations constructed by Yamaguchi in \cite{Ya17}.

\section{Corollaries of the main result}

In this section we obtain corollaries of the main result. First, we deduce the expression for the determinant of the twisted Laplacian in terms of the Selberg zeta function. Then, we turn our attention to two special cases of Theorem \ref{thm:main_theorem}: the case when the representation  $\rho$ is cyclic and the case when there are no torsion points on the surface.

\subsection{Determinant of the twisted Laplacian}

In this subsection, we obtain the formulas that express the regularized determinant
of the twisted Laplacian in terms of special values of the twisted Selberg zeta function assuming zero is
an eigenvalue of any multiplicity.
Quite simply, the corollaries are direct consequences of Theorem \ref{thm:main_theorem} by setting $s=1$.
Nonetheless, we are making the statements explicit for the convenience of the reader.

Let us recall that $M_\rho\geq 0$  denotes the multiplicity of the eigenvalue $\mu=0$ of the twisted Laplacian $\Delta_{\tau_m,\rho}^{\sharp}$.
By Theorem \ref{mero}, in view of the restriction that $m\in(-1,1]$,
it is straightforward to conclude that $s=1$ is the zero of $Z(s;\rho)$ of order $M_\rho$.
Additionally, $s=1$ is not a zero of the identity contribution $Z_I(s;\rho)$ nor a zero of the elliptic contribution $Z_{\rm ell}(s;\rho)$.

With this notation, in view of \eqref{eq. det final} and \eqref{eq. def ast det}, we have the following corollary.

\begin{coro}\label{detfinal}
We have the following evaluation of the determinant of the twisted Laplacian.
\begin{itemize}
  \item[(i)] If $M_\rho=0$, meaning that $\mu=0$ is not the eigenvalue of the twisted Laplacian, then $Z(s;\rho)$ is non-vanishing at $s=1$ and
  \begin{equation}\label{detatone1}
  \mathrm{det}\Delta_{\tau_m,\rho}^{\sharp}= Z(1;\rho)Z_I(1;\rho)Z_{\rm ell}(1;\rho)e^{\widetilde C}.
  \end{equation}

  \item[(ii)]  If $M_\rho>0$, then
\begin{align}\label{detatone2}
\mathrm{det}(\Delta^{\sharp\ast}_{\tau_m,\rho})&= \frac{1}{M_\rho!}\left.\frac{d^{M_\rho}}{ds^{M_\rho}} \left(Z(s;\rho)Z_I(s,\rho)Z_{\rm ell}(s;\rho)e^{\widetilde C}\right)\right|_{s=1}\\\notag&=\frac{1}{M_\rho!}\left.\frac{d^{M_\rho}}{ds^{M_\rho}} Z(s;\rho)\right|_{s=1}Z_I(1;\rho)Z_{\rm ell}(1;\rho)e^{\widetilde C}.
\end{align}
\end{itemize}
\end{coro}

We note that further simplifications of $Z_{\rm ell}(1;\rho)$ are readily attainable using the
multiplication formula for the Gamma function.

\subsection{Determinant of the twisted Laplacian associated to a cyclic representation}

If the representation is cyclic, meaning that $\rho(u)=I_n$, where $n=\dim(V_\rho)$, we deduce the following corollary

\begin{coro}
With the notation as above, for $j=1,\ldots, r$ and $\ell=0,\ldots,\nu_j-1$ let
$$
C_j(\ell;\rho):= \frac{1}{\nu_j}\sum_{k=1}^{\nu_j-1}\mathrm{Tr}(\rho(\gamma_j)^k)\frac{\sin\left(\frac{\pi k}{\nu_j}(2\ell+1)\right)}{\sin\left(\frac{\pi k}{\nu_j}\right)}.
$$
Then, the twisted determinant of the Laplacian $\mathrm{det}(\Delta_{\tau_0,\rho}^{\sharp} +s(s-1))$
and the twisted Selberg zeta function $Z(s;\rho)$ satisfy the relation
\begin{equation}\label{eq. det final 0}
\mathrm{det}(\Delta_{\tau_0,\rho}^{\sharp} +s(s-1))= Z(s;\rho)Z_I(s,\rho)Z_{\rm ell}(s;\rho)e^{{\widetilde C}_0},
\end{equation}
where
\begin{equation}\label{idcompact}
Z_{I}(s;\rho) = \exp\bigg\{2C_\rho(X)\bigg(s\log(2\pi) + s(1-s) + \log\Gamma\left(s\right) - 2\log G\left(s+1\right)\bigg)\bigg\},
\end{equation}
\begin{equation}\label{eq. Z ell m=0}
Z_{\rm ell}(s;\rho) = \prod_{j=1}^r \prod_{\ell=0}^{\nu_j-1}\Gamma\left(\frac{s+\ell}{\nu_j}\right)^{C_j(\ell;\rho)}
\end{equation}
and with the constant $\widetilde{C}_0$ given by
\begin{equation}\label{eq. tilde C0}
\widetilde{C}_0=-\dim(V_\rho)\chi(X)(2\zeta'(-1)-\log\sqrt{2\pi})+ \sum_{j=1}^{r}\frac{\log\nu_j}{2\nu_j}\sum_{k=1}^{\nu_j-1}\frac{\mathrm{Tr}
  (\rho(\gamma_j)^k)}{\sin^2\left(\frac{\pi k}{\nu_j}\right)}.
\end{equation}
\end{coro}
\begin{proof}
We want to prove that, after inserting $m=0$ in \eqref{eq. det final}, \eqref{e:ZI_def}, \eqref{eq. Z ell final} and \eqref{eq. tilde C defn} we deduce \eqref{eq. det final 0}.
Trivially, by inserting $m=0$ into \eqref{e:ZI_def} we get \eqref{idcompact}. Since $C_0(j\,\ell;\rho)+ \tilde C_0(j\,\ell;\rho)=C_j(\ell;\rho)$ for all $j=1,\ldots,r$, $\ell=0,\ldots,\nu_j-1$ by taking $m=0$ in \eqref{eq. Z ell final} we get \eqref{eq. Z ell m=0}. Finally, for $m=0$ the constant \eqref{eq. tilde C defn} becomes \eqref{eq. tilde C0}. The proof is complete.
\end{proof}

\subsection{Groups without torsion elements}

Using the results of the previous sections, we can deduce similar results in the case that we do not consider torsion elements in the lattice $\Gamma$.

\begin{coro}
Let $X=\Gamma\backslash \H^2$ be a compact hyperbolic surface without elliptic elements,
and $\rho$ an arbitrary representation of $\Gamma$ as above.
Then
\begin{equation}\label{detcompact}
\mathrm{det}(\Delta^{\sharp\ast}_{\tau_m,\rho} +s(s-1))= Z(s;\rho)Z_I(s,\rho)e^{\widetilde C},
\end{equation}
where $Z_I(s;\rho)$ is defined by \eqref{e:ZI_def} and
\begin{equation}\label{constantcompact}
\widetilde{C}=2C_\rho(X)(2\zeta'(-1)-\log\sqrt{2\pi}).
\end{equation}
\end{coro}

\begin{proof}
The assertion follows from \eqref{eq. det final} and \eqref{eq. tilde C defn} by taking $r=0$.
\end{proof}




\begin{coro}
Let $X=\Gamma\backslash \H^2$ be a compact hyperbolic surface without elliptic elements,
and $\rho$ an arbitrary representation of $\Gamma$ as above.
\begin{itemize}
  \item[(i)] If $M_\rho=0$, meaning that $\mu=0$ is not the eigenvalue of the twisted Laplacian, then $Z(s;\rho)$ is non-vanishing at $s=1$ and
  \begin{equation*}
  \mathrm{det}\Delta_{\tau_m,\rho}^{\sharp}= Z(1;\rho)Z_I(1;\rho)e^{\widetilde C}.
  \end{equation*}

  \item[(ii)]  If $M_\rho>0$, then
\begin{align}\label{detatone2}
\mathrm{det}(\Delta^{\sharp\ast}_{\tau_m,\rho})&= \frac{1}{M_\rho!}\left.\frac{d^{M_\rho}}{ds^{M_\rho}} \left(Z(s;\rho)Z_I(s,\rho)Z_{\rm ell}(s;\rho)e^{\widetilde C}\right)\right|_{s=1}\\\notag&=\frac{1}{M_\rho!}\left.\frac{d^{M_\rho}}{ds^{M_\rho}} Z(s;\rho)\right|_{s=1}Z_I(1;\rho)e^{\widetilde C}.
\end{align}
\end{itemize}
In these formulas, the constant $\widetilde C$ is given by \eqref{constantcompact}.
\end{coro}

\begin{proof}
The assertion follows from Corollary \ref{detfinal} by taking $r=0$.
\end{proof}

In all cases above, note that $\Gamma(1)=G(2)=1$, so then from \eqref{idcompact} we get that
$Z_I(1;\rho) =  \exp((2\log (2\pi))C_{\rho}(X))$.  Finally, the factor $Z_{\rm ell}(1;\rho)$
is explicit as well in terms of special values of the Gamma function.

\bibliographystyle{amsplain}
\bibliography{ref}

\vspace{5mm}\noindent
Jay Jorgenson \\
Department of Mathematics \\
The City College of New York \\
Convent Avenue at 138th Street \\
New York, NY 10031\\
U.S.A. \\
e-mail: jjorgenson@mindspring.com

\vspace{5mm}\noindent
Lejla Smajlovi\'c \\
Department of Mathematics \\
University of Sarajevo\\
Zmaja od Bosne 35, 71 000 Sarajevo\\
Bosnia and Herzegovina\\
e-mail: lejlas@pmf.unsa.ba

\vspace{5mm}\noindent
Polyxeni Spilioti\\
Department of Mathematics \\
University of Patras\\
Panepistimioupoli Patron 265 04\\
Greece\\
e-mail: pspilioti@upatras.gr

\end{document}